\documentclass[a4paper,12pt]{amsart}

\usepackage{amsmath}
\usepackage{amssymb}
\usepackage{latexsym}
\usepackage[mathscr]{eucal}
\usepackage{graphics}
\usepackage{multicol}
\usepackage{braket}
\usepackage{longtable}

\usepackage{newtxtext, newtxmath}

\usepackage{ascmac}
\usepackage{framed}
\usepackage{ulem}
\usepackage{cancel}
\usepackage{arydshln}
\usepackage{bm}
\allowdisplaybreaks
\newcommand{\relmiddle}[1]{\mathrel{}\middle#1\mathrel{}}

\theoremstyle{plain}
\newtheorem{thm}{\sc\ Theorem}[section]
\newtheorem{lem}[thm]{\sc Lemma}
\newtheorem{prop}[thm]{\sc Proposition}
\newtheorem{corollary}[thm]{\sc Corollary}

\makeatletter
\def\thefootnote{\ifnum\c@footnote>\z@\leavevmode\lower.5ex%
	\hbox{$^{\@arabic\c@footnote)}$}\fi}
\makeatother

\makeatletter
\def\section{\@startsection{section}{1}%
	\z@{.7\linespacing\@plus\linespacing}{.5\linespacing}%
	{\normalfont\bfseries\raggedright}}

\def\subsection{\@startsection{subsection}{2}%
	\z@{.5\linespacing\@plus.7\linespacing}{.3\linespacing}%
	{\normalfont\bfseries\raggedright}}
\makeatother

\theoremstyle{definition}

\theoremstyle{remark}

\theoremstyle{definition} 

\def\dfrac#1#2{\displaystyle \frac{#1}{#2}}

\def\det{\mbox{\rm {det}}}
\def\diag{\mbox{\rm {diag}}}

\def\Der{\mbox{\rm {Der}}}

\def\ad{\mbox{\rm {ad}}}

\def\Iso{\mbox{\rm {Iso}}}

\def\Ker{\mbox{\rm {Ker}}}

\def\ov{\overline}

\def\dfrac#1#2{\displaystyle \frac{#1}{#2}}

\def\R{\mbox{\boldmath $R$}}

\def\Z{\mbox{\boldmath $Z$}}

\def\0{\mbox{\boldmath {0}}}
\def\1{\mbox{\boldmath {1}}}
\def\2{\mbox{\boldmath {2}}}
\def\3{\mbox{\boldmath {3}}}
\def\4{\mbox{\boldmath {4}}}
\def\5{\mbox{\boldmath {5}}}
\def\6{\mbox{\boldmath {6}}}
\def\7{\mbox{\boldmath {7}}}
\def\8{\mbox{\boldmath {8}}}
\def\9{\mbox{\boldmath {9}}}

\usepackage{epic,eepic}

\begin{document}

\title[Realizations of inner automorphisms of order four on $E_8$
Part III]{Realizations of inner automorphisms \\
 of order four and fixed points subgroups \\
 by them on the connected compact \\
 exceptional Lie group $E_8$,  Part III}

\author[Toshikazu Miyashita]{\vspace{-4mm}}

\dedicatory{\rm By  \\ \vspace{3mm} Toshikazu {\sc Miyashita} }





\address{1365-3 Bessho onsen    \endgraf
	Ueda City                     \endgraf
	Nagano Prefecture 386-1431    \endgraf
	Japan}
\email{anarchybin@gmail.com}

\begin{abstract}
The compact simply connected Riemannian 4-symmetric spaces were
classified by J.A. Jim{\'{e}}nez according to type of the Lie algebras.
As homogeneous manifolds, these spaces are of the
form $G/H$, where $G$ is a connected compact simple Lie group with
an automorphism $\tilde{\gamma}$ of order four on $G$ and  $H$ is a
fixed points subgroup $G^\gamma$ of $G$. According to the
classification by J.A. Jim{\'{e}}nez, there exist seven compact
simply connected Riemannian 4-symmetric spaces $ G/H $ in the case where $ G $ is of type $ E_8 $. In the present article,
we give the explicit form of automorphisms  $\tilde{\omega}^{}_4,\tilde{\kappa}^{}_4$  and $ \tilde{\varepsilon}^{}_4$ of order four on $E_8$ induced by the
$C$-linear transformations $\omega^{}_4, \kappa^{}_4$  and $
\varepsilon^{}_4$ of the 248-dimensional $ C $-vector
space ${\mathfrak{e}_8}^{C}$, respectively.
Further, we determine the structure of these fixed points subgroups
$(E_8)^{{}_{\omega^{}_4}}, (E_8)^{{}_{\kappa^{}_4}}$ and $(E_8)^{{}
	_{\varepsilon^{}_4}}$ of $ E_8 $. These amount to the global realizations of three spaces among seven Riemannian 4-symmetric spaces $ G/H $ above corresponding to the Lie algebras $ \mathfrak{h}=\mathfrak{su}(2)\oplus i\bm{R} \oplus
\mathfrak{e}_6, i\bm{R} \oplus \mathfrak{so}(14)$ and $\mathfrak{h}
=\mathfrak{su}(2) \oplus i\bm{R}\oplus \mathfrak{so}(12)$, where $ \mathfrak{h}={\rm Lie}(H) $. With this article, the all realizations of inner automorphisms of order four and fixed points subgroups by them have been completed in $ E_8 $.
\end{abstract}

\maketitle

\section {Introduction}
Let $G$ be a Lie group and $H$ a compact subgroup of $G$. A
homogeneous space $G/H$ with $G$-invariant Riemannian metric $g$ is
called a Riemannian $4$-symmetric space if there exists an
automorphism $\tilde{\gamma}$ of order four on $G$ such that $({G^
	\gamma})_{0} \subset H \subset G^\gamma$, where $G^\gamma$ and $({G^
	\gamma})_{0}$ are the fixed points subgroup of $G$ by $
\tilde{\gamma}$ and its identity component, respectively.

Now, for the exceptional compact Lie group of type $E_8$, as in Table
below, there exist seven cases of the compact simply connected
Riemannian $4$-symmetric spaces which were classified by J.A.
Jim\'{e}nez as mentioned in abstract (\cite{Jim}). Accordingly, our
interest is to realize the groupfication for the classification as
Lie algebra.

Our results of groupfication corresponding to the Lie algebra $
\mathfrak{h}$ in Table are given as follows.

\begin{longtable}[c]{clll}
	\noalign{\hrule height 1pt}
	\noalign{\vspace{-1pt}}
	&&&\\[-6pt]
	Case & \hspace*{7mm}  $\mathfrak{h}$  &  $\tilde{\gamma}
	$  &  \hspace*{15mm}  $H=G^\gamma$
	\\
	\noalign{\vspace{3pt}}
	\noalign{\hrule height 0.5pt}
	&&&\\[-8pt]
	1
	& $\mathfrak{so}(6) \oplus \mathfrak{so}(10)$
	& $\tilde{\sigma}'_{{}_4}$
	& $(Spin(6) \times Spin(10))/\bm{Z}_4$
	\\[6pt]
	2
	& $i\bm{R} \oplus \mathfrak{su}(8)$
	& $\tilde{w}_{{}_4}$
	& $(U(1) \times SU(8))/\bm{Z}_{24}$
	\\[6pt]
	3
	& $i\bm{R} \oplus \mathfrak{e}_7 $
	& $\tilde{\upsilon}_{{}_4}$
	& $(U(1) \times E_7)/\bm{Z}_2$
	\\[6pt]
	4
	& $\mathfrak{su}(2) \oplus \mathfrak{su}(8)$
	& $\tilde{\mu}_{{}_4}$
	& $(SU(2) \times SU(8))/\bm{Z}_4$
	\\[6pt]
	5
	& $\mathfrak{su}(2) \oplus i\bm{R} \oplus \mathfrak{e}_6$
	& $\tilde{\omega}_{{}_4}$
	& $(SU(2) \times U(1) \times E_6)/(\bm{Z}_2 \times \bm{Z}_3)$
	\\[6pt]
	6
	& $i\bm{R} \oplus \mathfrak{so}(14)$
	& $\tilde{\kappa}_{{}_4}$
	& $(U(1) \times Spin(14))/\bm{Z}_4$
	\\[6pt]
	7
	& $\mathfrak{su}(2) \oplus i\bm{R} \oplus \mathfrak{so}(12)$
	& $\tilde{\varepsilon}_{{}_4}$
	& $(SU(2) \times  U(1) \times Spin(12))/(\bm{Z}_2 \times \bm{Z}_2)$
	\\[6pt]
	\noalign{\hrule height 1pt}
\end{longtable}


In \cite{miya2} and \cite{miya3}, the author has already realized the groupfication for Case $ 1 $ and Cases $ 2,3,4 $ in Table, respectively.
In the present article, we state the realizations of the
group $H$ for Cases 5,6 and Case 7.

Finally, the author would like to say that the feature of this
article is to give elementary proofs of the isomorphism of groups
by using the homomorphism theorem, and
the all global realizations have been completed in $ E_8 $.

This article is a continuation of \cite{miya3}, hence we start from Section $ 7 $. The readers refer to \cite{miya2} for preliminary results and also to \cite{miya2},\cite{miya3},\cite{miya0} or \cite{iy0} for notations.
Note that we change the numbering of Case 5 and Case 6 in
\cite{miya2} to the numbering of Case 3 and Case 4 in the present
article, respectively.

\setcounter{section}{6}

\section { Case 5. The automorphism $\tilde{\omega}^{}_4$ of
order four and the group $(E_8)^{\omega^{}_4}$}

In this section (also in Section $ 8 $ and Section $ 9 $), we use the $248$-dimensional $ C $-vector space ${\mathfrak{e}_8}^{C}$ used in Case 1 (\cite{miya2}) and the simply connected compact exceptional Lie group of type $E_8$ constructed by T. Imai and I. Yokota (\cite{Imai}).

We define a $ C $-linear transformation $\omega^{}_4$ of $
{\mathfrak{e}_8}^C$ by
$$
\omega^{}_4(\varPhi, P, Q, r, s, t)=(\iota\varPhi \iota^{-1}, -\iota
P, -\iota Q, r, s,  t),
$$
where $\iota$ on the right hand side is the $
C$-linear transformation of $\mathfrak{P}^C$ defined in \cite [Definition of Subsection 4.10 (p.131)]{iy0} (the space  $ \mathfrak{P}^C $, called the Freudenthal $ C $-vector space, is defined in \cite[Preliminaries (p.94)]{miya2}).
Note that $ \omega^{}_4 $ is the composition transformation of $ \iota, \upsilon \in E_7 \subset E_8 $, where $\upsilon$ is the $
C$-linear transformation of ${\mathfrak{e}_8}^{C}$ defined in \cite[Definition of Subsection 5.7(p.174)]{iy0}.
Moreover since $ \iota,\upsilon $ are expressed as elements of $ E_8 $ by
\begin{align*}
    \iota&=\exp\left( \frac{2\pi i}{4}\ad(\varPhi(0,0,0,3),0,0,0,0,0)\right),
    \\
    \upsilon&=\exp\left( \frac{2\pi i}{4}\ad(\varPhi(0,0,0,6),0,0,0,0,0)\right),
\end{align*}
respectively and together with $ [\varPhi(0,0,0,3), \varPhi(0,0,0,6)]=0 $, we have
\begin{align*}
    \omega^{}_4=\exp\left( \frac{2\pi i}{4}\ad(\varPhi(0,0,0,9),0,0,0,0,0)\right).
\end{align*}
Hence it follows from above that $\omega^{}_4 \in E_8 $ and $ (\omega^{}_4)^4=1, (\omega^{}_4)^2=\upsilon$, so that
$\omega^{}_4 $ induces the inner automorphism
$\tilde{\omega}^{}_4$ of order four on $E_8$: $\tilde{\omega}^{}_4(\alpha)=\omega^{}_4 \alpha {\omega^{}_4}^{-1}, \alpha \in E_8$.
\vspace{2mm}

Now, we will study the subgroup $(E_8)^{\omega^{}_4}$ of $ E_8 $:
\begin{align*}
(E_8)^{\omega^{}_4}=\left\lbrace  \alpha \in E_8  \relmiddle{|}
\omega^{}_4 \alpha=\alpha \omega^{}_4 \right\rbrace .
\end{align*}

The aim of this section is to determine the structure of the group $(E_8)^{\omega^{}_4}$.
Before that, we prove the lemma and propositions needed later.

\begin{lem}\label{lem 7.1}
    The Lie algebra $(\mathfrak{e}_8)^{\omega^{}_4}$ of the group
    $(E_8)^{\omega^{}_4}$ is given by
    \begin{align*}
    (\mathfrak{e}_8)^{\omega^{}_4}&=\left\lbrace \ad(R) \in \Der(\mathfrak{e}_8)\relmiddle{|} \omega^{}_4\ad(R)=\ad(R)\omega^{}_4, R \in \mathfrak{e}_8 \right\rbrace
    \\
    &\cong \left\lbrace  R \in  \mathfrak{e}_8 \relmiddle{|}
    \omega_4 R=R \right\rbrace
    \\
    &=\left\lbrace R=(\varPhi, 0,0,r, s,-\tau s) \relmiddle{|} \varPhi \in
    (\mathfrak{e}_7)^{\iota} \cong \mathfrak{u}(1) \oplus \mathfrak{e}_6, r \in i\R, s \in C \right\rbrace ,
    \end{align*}
where $ (\mathfrak{e}_7)^{\iota} $ is the Lie algebra of the group $ (E_7)^\iota:=\{\alpha \in E_7 \,|\, \iota\alpha=\alpha\iota \} $.

\end{lem}

In particular, we have $ \dim((\mathfrak{e}_8)^{\omega^{}_4})=(1+78)+1+2=82 $.
\begin{proof}
    By doing straightforward computation, we can obtain the required results. The isomorphism $ (\mathfrak{e}_7)^{\iota} \cong \mathfrak{u}(1) \oplus \mathfrak{e}_6 $ as Lie algebras follows from \cite[Theorem 4.10.2]{iy0}.
\end{proof}

\begin{prop}\label{prop 7.2}
    The group $ (E_8)^{\omega^{}_4} $ contains the group $ (E_7)^\iota $
    which is isomorphic to the group $ (U(1)\times E_6)/\Z_3,\Z_3=\{(1,1),(\omega, \allowbreak \phi(\omega^2)),(\omega^2,\phi(\omega))\} ${\rm :} $ (E_8)^{\omega^{}_4} \supset (E_7)^\iota \cong (U(1)\times E_6)/\Z_3 $, where $ \omega=(-1/2)+(\sqrt{3}/2)i $.
\end{prop}
\begin{proof}
    Immediately, we have $ (\mathfrak{e}_7)^\iota \subset (\mathfrak{e}_8)^{\omega^{}_4} $ from Lemma \ref{lem 7.1}. Moreover since
    both of the groups $ (E_8)^{\omega^{}_4} $ and $ (E_7)^\iota $ are connected, we have $ (E_7)^\iota \subset (E_8)^{\omega^{}_4} $.
    However, we will prove it directly.
    Let $\alpha \in(E_7)^{\iota}$.
    Note that $-1 \in z(E_7)$ (the center of $E_7$), it follows that
    \begin{align*}
        \omega^{}_4 \alpha(\varPhi, P, Q, r, s, t)&=\omega^{}_4
        (\alpha\varPhi\alpha^{-1}, \alpha P, \alpha Q, r, s, t)\\
        &=(\iota \alpha \varPhi \alpha^{-1} \iota^{-1}, -\iota \alpha
        P, -\iota \alpha Q, r, s, t)\\
        &=(\alpha(\iota \varPhi \iota^{-1})\alpha^{-1}, \alpha(-\iota
        P), \alpha(-\iota  Q), r, s, t)\\
        &=\alpha\omega^{}_4 (\varPhi, P, Q, r, s, t),\,\,(\varPhi, P, Q, r, s, t) \in \mathfrak{e}_8^C,
    \end{align*}
    that is, $\omega^{}_4 \alpha=\alpha\omega^{}_4 $.
    Hence we have $ \alpha \in (E_8)^{\omega^{}_4} $, so that the first half is proved.

    As for the proof of the second half, we define a mapping $ \varphi_\iota:U(1)\times E_6 \to (E_7)^\iota $ by
    \begin{align*}
        \varphi_\iota(\theta,\beta)=\phi(\theta)\beta,
    \end{align*}
    where the mapping $ \phi:U(1)=\{\theta \in C \,|\, (\tau \theta)\theta=1\} \to E_7 $ is defined by $ \phi(\theta)(X,Y,\xi,\eta)\allowbreak =(\theta^{-1}X,\theta Y, \theta^3 \xi,\theta^{-3}\eta), (X,Y,\xi,\eta) \in \mathfrak{P}^C $.

    Then the mapping $ \varphi_\iota $ induces the required isomorphism (see \cite[Theorem 4.10.2]{iy0} in detail).
\end{proof}

Let the special unitary group $ SU(2)= \{ A \in M(2, C)\, | \, (\tau\,{}^t\!A)A = E, \det A = 1 \} $ and we define an embedding $ \phi_\upsilon: SU(2) \to E_8 $ by
\begin{align*}
\phi_\upsilon(\begin{pmatrix}
		a & -\tau b \\
		b & \tau a
	\end{pmatrix}) =
\begin{pmatrix}
  1 & 0 & 0 & 0 & 0 & 0
\vspace{1mm}\cr
                   0 & \tau a1 &  b1 & 0 & 0 & 0
\vspace{1mm}\cr
                   0 & -\tau b1 &  a1 & 0 & 0 & 0
\vspace{1mm}\cr
                   0 & 0 & 0 & (\tau a)a - (\tau b)b & (\tau a)(\tau b) & ab
\vspace{1mm}\cr
                   0 & 0 & 0 & -2(\tau a)b & (\tau a)^2 & -b^2
\vspace{1mm}\cr
                   0 & 0 & 0 & 2a(\tau b) & -(\tau b)^2 & a^2
\end{pmatrix}.
\end{align*}

\begin{prop}{\cite[Theorem 5.7.4]{iy0}}\label{prop 7.3}
	The group $(E_8)^{\omega_4}$ contains a subgroup
	$$
	\phi_\upsilon(SU(2)) = \{ \phi_\upsilon(A) \in E_8 \, | \,
	A \in SU(2) \}
	$$
	which is isomorphic to the group $SU(2) = \{ A \in M(2, C)
	\, | \, (\tau\,{}^t\!A)A = E, \det A = 1 \}$.
\end{prop}
\begin{proof}
	For $A = \begin{pmatrix}
		a & -\tau b \\
		b & \tau a
	\end{pmatrix}:=
	\exp\begin{pmatrix}
		- i\nu & - \tau\varrho \\
		\varrho & i\nu
	\end{pmatrix} \in SU(2)$, where $ \begin{pmatrix}
		- i\nu & - \tau\varrho \\
		\varrho & i\nu
	\end{pmatrix}\in \mathfrak{su}(2) $, we have $\phi_\upsilon(A) = {\rm
		\exp}({\rm ad}(0, 0, 0, i\nu, $ $\varrho, -\tau\varrho))
	\in
	(E_8)^{\omega_4}$ (Lemma \ref{lem 7.1}).
\end{proof}
\noindent Note that there are some errata in definition of $ \varphi_3(A) $ of \cite[Theorem 5.7.4]{iy0}.

Here we need the following result in the proof of Theorem \ref{thm 7.5} below.

\begin{thm}{\rm (\cite[Theorem 5.7.6]{iy0})}\label{thm 7.4}
   The group $ (E_8)^\upsilon $ is isomorphic to the group $ (SU(2) \times E_7)/\Z_2, \Z_2=\{(E,1),(-E,-1)\} ${\rm :} $ (E_8)^\upsilon \cong (SU(2) \times E_7)/\Z_2 $.
\end{thm}
\begin{proof}
   We define a mapping $ \varphi:SU(2) \times E_7 \to (E_8)^\upsilon $ by
   \begin{align*}
      \varphi(A,\delta)=\phi_\upsilon(A)\delta.
   \end{align*}

Then the mapping $ \varphi $ induces the required isomorphism.
\end{proof}

Now, we determine the structure of the group $(E_8)^{\omega_4}$.

\begin{thm}\label{thm 7.5}
    The group $(E_8)^{\omega_4}$ is isomorphic to the group $(SU(2)
    \times U(1) \times E_6)/(\Z_2 \allowbreak \times \Z_3),
    \Z_2=\{(E,1,1),(-E, -1, 1) \} ,
    \Z_3=\{(E,1,1), (E,\omega, \phi(\omega^2)), (E, \omega^2, \phi(\omega)) \}${\rm :} \allowbreak $(E_8)^{\omega_4} \cong  (SU(2)
    \times U(1) \times E_6)/(\Z_2 \times \Z_3)$.
\end{thm}
\begin{proof}
    We define a mapping $\varphi_{{}_{\omega^{}_4}}: SU(2) \times U(1)
    \times E_6 \to (E_8)^{\omega^{}_4}$ by
    \begin{align*}
      \varphi_{{}_{\omega^{}_4}}(A, \theta,
      \beta)=\phi_\upsilon(A)\varphi_\iota(\theta,\beta)(=\varphi(A,\varphi_\iota(\theta,\beta))),
    \end{align*}
where $ \varphi $ is defined in the proof of Theorem \ref{thm 7.4}.

First, it is clear that $ \varphi_{{}_{\omega^{}_4}} $ is well-defined from Propositions \ref{prop 7.2}, \ref{prop 7.3} and since the mapping $ \varphi_{{}_{\omega^{}_4}} $ is the restriction of the mapping $ \varphi:SU(2)\times E_7 \to (E_8)^\upsilon $, $\varphi_{{}_{\omega^{}_4}}$ is a homomorphism.


    Next, we will prove that $\varphi_{{}_{\omega^{}_4}}$ is surjective. Let $\alpha \in (E_8)^{\omega^{}_4}$. From $(\omega^{}_4)^2=\upsilon$, we easily see
    $ (E_8)^{\omega^{}_4} \subset (E_8)^\upsilon $. Hence there exist
    $A \in SU(2)$ and $\delta \in E_7$ such that
    $\alpha=\varphi(A, \delta)$ (Theorem \ref{thm 7.4}).
    Moreover, from the condition
    $\omega^{}_4 \alpha {\omega^{}_4}^{-1}=\alpha $, that is, $\omega^{}_4
    \varphi(A, \delta) {\omega^{}_4}^{-1}=\varphi(A, \delta)$,
    we have $\varphi(A, \iota \delta \iota^{-1})=\varphi(A,
    \delta)$.
    Indeed,
note that the formula $\omega^{}_4 \delta{\omega^{}_4}^{-1}=\iota\delta\iota^{-1}$ follows from $\delta \in E_7$, so that it follows from
$ \omega^{}_4 \phi_\upsilon(A) {\omega^{}_4}^{-1}=\phi_\upsilon (A) $ (Proposition \ref{prop 7.3}) that
    \begin{align*}
    \omega^{}_4 \varphi(A, \delta)
    {\omega^{}_4}^{-1}&=\omega^{}_4 (\phi_\upsilon(A)\delta
    ){\omega^{}_4}^{-1}\\
    &=(\omega^{}_4 \phi_\upsilon(A){\omega^{}_4}^{-1})(\omega^{}_4
    \delta{\omega^{}_4}^{-1})\\
    &=\phi_\upsilon(A)(\iota\delta\iota^{-1})\\
    &=\varphi(A, \iota\delta \iota^{-1} ),
    \end{align*}
    that is,  $\varphi(A, \iota \delta \iota^{-1})=\varphi(A, \delta)$.

    Thus, since $ \Ker\varphi=\{(E,1),(-E,-1)\} $ (Theorem \ref{thm 7.4}), we have the following
    $$
    \left\{\begin{array}{l}
    A = A
    \vspace{1mm}\\
    \iota\delta \iota^{-1}= \delta
    \end{array} \right.
    \qquad   \text{or}\qquad
    \left\{\begin{array}{l}
    A= -A
    \vspace{1mm}\\
    \iota\delta \iota^{-1} = -\delta.
    \end{array} \right.
    $$
    In the latter case, this case is impossible because of $A
    \not=O$, where $ O $ is the zero matrix.

    \noindent In the former case, $ \delta \in (E_7)^\iota $ follows from the second condition, so that there exist $\theta \in U(1)$ and $\beta \in E_6$ such that $\delta=\varphi_\iota(\theta,\beta)$ (Proposition \ref{prop 7.2}).
    Hence there exist $ A \in SU(2), \theta \in U(1) $ and $ \beta \in E_6 $  such that $ \alpha=\varphi(A,\varphi_\iota(\theta,\beta))=\varphi_{{}_{\omega^{}_4}}(A,\theta,\beta) $. The proof of surjective is completed.

    Finally, we will determine $ \Ker\,\varphi_{{}_{\omega^{}_4}} $.
    From the definition of kernel, we have
    \begin{align*}
    \Ker\,\varphi_{{}_{\omega^{}_4}}&=\{(A, \theta, \beta) \in
    SU(2) \times U(1) \times E_6 \,|\,
    \varphi_{\omega_4}(A,\theta, \beta)=1  \}\\
    &=\{(A, \theta, \beta) \in SU(2)  \times U(1) \times E_6
    \,|\, \varphi(A,\varphi_\iota(\theta,\beta))=1 \}.
    \end{align*}
    Here, the mapping $\varphi_{{}_{\omega^{}_4}}$ is the
    restriction of the mapping $\varphi $ and together with $ \Ker\,\varphi=\{(E,1), (-E, -1) \}$ (Theorem \ref{thm 7.4}), we will find the elements $ (A,\theta,\beta) \in SU(2) \times U(1) \times E_6 $ satisfying the
    following
    $$
    \left\{\begin{array}{l}
    A = E
    \vspace{1mm}\\
    \varphi_\iota(\theta,\beta) = 1
    \end{array} \right.
    \qquad   \text{or}\qquad
    \left\{\begin{array}{l}
    A= -E
    \vspace{1mm}\\
    \varphi_\iota(\theta,\beta)= -1.
    \end{array} \right.
    $$
    In the former case, from $ \Ker\,\varphi_\iota=\{(1,1),(\omega,
    \phi(\omega^2)),(\omega^2, \phi(\omega))\}$ (Proposition \ref{prop 7.2}), we have the following
    $$
    \left\{\begin{array}{l}
    A = E
    \vspace{1mm}\\
    \theta=1
    \vspace{1mm}\\
    \beta=1,
    \end{array} \right.  \qquad
    \left\{\begin{array}{l}
    A = E
    \vspace{1mm}\\
    \theta=\omega
    \vspace{1mm}\\
    \beta=\phi(\omega^2)
    \end{array} \right.
    \qquad   \text{or}    \qquad
    \left\{\begin{array}{l}
    A = E
    \vspace{1mm}\\
    \theta=\omega^2
    \vspace{1mm}\\
    \beta=\phi(\omega).
    \end{array} \right.
    $$
    In the latter case, the second
    condition can be rewritten as $\phi(-\theta)\beta= 1$ from $-1=\phi(-1)$. Hence, as in the former case, we have the following
    $$
    \left\{\begin{array}{l}
    A = -E
    \vspace{1mm}\\
    \theta=-1
    \vspace{1mm}\\
    \beta=1,
    \end{array} \right.  \qquad
    \left\{\begin{array}{l}
    A = -E
    \vspace{1mm}\\
    \theta=-\omega
    \vspace{1mm}\\
    \beta=\phi(\omega^2)
    \end{array} \right.
    \qquad   \text{or}    \qquad
    \left\{\begin{array}{l}
    A =- E
    \vspace{1mm}\\
    \theta=-\omega^2
    \vspace{1mm}\\
    \beta=\phi(\omega).
    \end{array} \right.
    $$
    Hence we can obtain
    \begin{align*}
    \Ker\,\varphi_{{}_{\omega^{}_4}}&=\left\lbrace
    \begin{array}{l}
    (E,1,1), (E,\omega, \phi(\omega^2)), (E, \omega^2 ,\phi(\omega)), \\
    (-E,-1,1),(-E,-\omega, \phi(\omega^2)),  (-E, -\omega^2 ,\phi(\omega))
    \end{array} \right\rbrace
    \\
    &=\{(E,1,1),(-E, -1, 1) \} \times \{(E,1,1), (E,\omega, \phi(\omega^2)), (E, \omega^2 ,\phi(\omega)) \}\\
    & \cong \Z_2 \times \Z_3.
    \end{align*}

    Therefore we have the required isomorphism
\begin{align*}
    (E_8)^{\omega^{}_4}
    \cong  (SU(2) \times U(1) \times E_6)/(\Z_2 \times \Z_3)
\end{align*}
\end{proof}


\section{Case 6. The automorphism $\tilde{\kappa}^{}_4$ of order
four and the group $(E_8)^{{}_{\kappa^{}_4}}$}

In this section, our aim is to prove main theorem: $(E_8)^{\kappa^{}_4}
\cong (U(1) \times Spin(14))/\Z_4$. In
\cite[Theorem 5.10]{miya0}, the author has proved that the
group $({E_8}^C)^{\kappa^{}_3}$ is isomorphic to the group $(C^*
\times Spin(14,C))/\Z_4$ using the
automorphism ${\tilde{\kappa}^{}_3}$ of order three on ${E_8}^C$: $({E_8}
^C)^{\kappa^{}_3} \cong (C^* \times Spin(14,C))/\Z_4$, where the group $ ({E_8}^C)^{\kappa^{}_3} $ is defined as $ ({E_8}^C)_0 $ in \cite[Theorem 5.10]{miya0} and $ \tilde{\kappa}^{}_3,\tilde{\kappa}^{}_4 $ are defined later. As in the proof of \cite[Theorem 5.10]{miya0}, using the
automorphism ${\tilde{\kappa}^{}_4}$ of order four on ${E_8}^C$, we first will prove that the group $({E_8}^C)^{\kappa^{}_4}$ is isomorphic to the group $(C^* \times Spin(14,C))/\Z_4$ in Subsection $ 8.1 $: $ ({E_8}^C)^{\kappa^{}_4} \cong (C^* \times Spin(14,C))/\Z_4 $, and using the result above, the structure of the group $ (E_8)^{\kappa^{}_4} $ will be determined in Subsection $ 8.2 $.
We often will use the results obtained and definitions in \cite[Section 5.3 ]{miya0}, then note that we change several signs used in \cite[Section 5.3]{miya0}. For example, $ \kappa^{}_3 $ in \cite[Section 5.3 (p.37)]{miya0} is changed to $ \nu^{}_3 $ in this section.

At the end of the preface of this section, since the content of this section is related to the subalgebra $ \mathfrak{g}_0 $ of the simple graded Lie algebras $ \mathfrak{g} $ of second kind
\begin{align*}
\mathfrak{g}=\mathfrak{g}_{-2}\oplus\mathfrak{g}_{-1}\oplus\mathfrak{g}_{0}\oplus\mathfrak{g}_{1}\oplus\mathfrak{g}_{2},\;\;[\mathfrak{g}_k, \mathfrak{g}_l] \subset \mathfrak{g}_{k+l},
\end{align*}
we will review about that. In a simple graded Lie algebra $ \mathfrak{g} $ of second kind, we know that there exists $ Z \in \mathfrak{g}_0 $, called a characteristic element, such that each $ \mathfrak{g}_k $ is the $ k $-eigenspace of $ \ad\, Z: \mathfrak{g} \to \mathfrak{g} $, so that
\begin{align*}
\mathfrak{g}_0=\left\lbrace X \in \mathfrak{g} \relmiddle{|} (\ad\,Z) X=0 \right\rbrace.
\end{align*}
Here, set $ \tilde{\gamma}_3:=\exp\left(\dfrac{2\pi i}{3}\ad\, Z \right)  $ as a inner automorphism of order three on $ \mathfrak{g} $, then we have
\begin{align*}
\mathfrak{g}_0=\left\lbrace X \in \mathfrak{g} \relmiddle{|} \tilde{\gamma}_3 X=X \right\rbrace=:(\mathfrak{g})^{\gamma_3}.
\end{align*}
Indeed, let $ X \in \mathfrak{g}_0 $. Then it follows from $ (\ad\,Z) X=0 $ that
\begin{align*}
\tilde{\gamma}_3 X&=\left( \exp\left(\dfrac{2\pi i}{3}\ad\, Z \right)\right)X =\left(\sum\limits_{n=0}^{\infty}\dfrac{1}{n!} \left(\dfrac{2\pi i}{3}\ad\, Z \right)^n\right)X=X.
\end{align*}
Hence we have $ X \in (\mathfrak{g})^{\gamma_3} $. Conversely, let $ X \in (\mathfrak{g})^{\gamma_3} $. Then we easily see $ X \in \mathfrak{g}_k $ for some $ k \in \{-2,-1,0,1,2\} $, so that $ (\ad\,Z)X=kX $ holds. Hence it follows from
\begin{align*}
X&=\tilde{\gamma}_3 X=\left( \exp\left(\dfrac{2\pi i}{3}\ad\, Z \right)\right)X=\left(\sum\limits_{n=0}^{\infty}\dfrac{1}{n!} \left(\dfrac{2\pi i}{3}\ad\, Z \right)^n\right)X
\\
&=\left(\sum\limits_{n=0}^{\infty}\dfrac{1}{n!} \left(\dfrac{2\pi i}{3}k \right)^n\right)X=\left( \exp \left( \dfrac{2\pi i}{3} k\right)  \right)X
\end{align*}
that  $ \exp \left( \dfrac{2\pi i}{3} k\right) =1 $. Thus we have $ k=3n, n \in \Z $, however since $ |k| \leq 2 $, we see $ k=0 $, that is, $ X \in \mathfrak{g}_0 $. With above, $ \mathfrak{g}_0=(\mathfrak{g})^{\gamma_3} $ is proved. This result will be useful later.

\subsection{The group $({E_8}^C)^{\kappa^{}_4}$}

We define $C$-linear transformations $\kappa^{}_3$ and $
\kappa^{}_4$ of ${\mathfrak{e}_8}^C$ by
\begin{align*}
 \kappa^{}_3({\varPhi}, P, Q, r, s, t) &=
 (\nu^{}_3{\varPhi}{\nu^{}_3}^{-1},
 \omega^2\nu^{}_3P, \omega\nu^{}_3Q, r, \omega s, \omega^2 t),
\\
\kappa^{}_4({\varPhi}, P, Q, r, s, t) &=
(\nu^{}_4{\varPhi}{\nu^{}_4}^{-1},-i\nu^{}_4 P, i\nu^{}_4Q, r, - s, -t),
\end{align*}
where $ \omega=(-1/2)+(\sqrt{3}/2)i $ and
$\nu^{}_3, \nu^{}_4 \in E_7$ on the right hand side are defined by
\begin{align*}
\nu^{}_3(X, Y, \xi, \eta)
&=(
\begin{pmatrix}
\omega^2 \xi_1         & x_3            & \ov{x}_2 \\
\ov{x}_3 & \omega \xi_2       & \omega x_1\\
x_2         & \omega \ov{x}_1 & \omega \xi_3
\end{pmatrix},
\begin{pmatrix}
\omega \eta_1       & y_3            & \ov{y}_2 \\
\ov{y}_3 & \omega^2 \eta_2       & \omega^2 y_1      \\
y_2         & \omega^2 \ov{y}_1   &\omega^2 \eta_3
\end{pmatrix},
\omega^2 \xi, \omega \eta ),
\\
\nu^{}_4(X, Y, \xi, \eta)
&=(
\begin{pmatrix}
-i \xi_1         & x_3            & \ov{x}_2 \\
\ov{x}_3 & i \xi_2       & i x_1\\
x_2         & i\, \ov{x}_1 & i \xi_3
\end{pmatrix},
\begin{pmatrix}
i \eta_1       & y_3            & \ov{y}_2 \\
\ov{y}_3 & -i \eta_2       & -i y_1      \\
y_2         & -i\, \ov{y}_1   &-i \eta_3
\end{pmatrix},
-i \xi, i \eta ), (X,Y,\xi,\eta) \in \mathfrak{P}^C,
\end{align*}
respectively.
Then $ \kappa^{}_3, \kappa^{}_4 $ can be expressed by
\begin{align*}
\kappa^{}_3=\exp\left( \frac{2\pi i}{3}\ad\,\kappa\right) ,
\;\;
\kappa^{}_4=\exp\left( \frac{2\pi i}{4}\ad\,\kappa\right),
\end{align*}
respectively, where $\kappa:=(\varPhi(-2E_1 \vee E_1, 0,0, -1), 0,0,-1,0,0)
\in {\mathfrak{e}_8}^C, E_1 \vee E_1:=(1/3)(2E_1-E_2-E_3)^\sim \in {\mathfrak{e}_6}^C$. Besides, $ \nu^{}_3, \nu^{}_4 $ can be also expressed by
\begin{align*}
    \nu^{}_3=\exp\left( \frac{2\pi i}{3}\varPhi(-2E_1 \vee E_1, 0,0, -1)\right),
    \;\;
    \nu^{}_4=\exp\left( \frac{2\pi i}{4}\varPhi(-2E_1 \vee E_1 0,0, -1)\right),
\end{align*}
respectively.
Hence it follows from above that $\kappa^{}_3, \kappa^{}_4 \in  E_8 \subset
{E_8}^C$ and $ (\kappa^{}_3)^3=(\kappa^{}_4)^4=1$, so that
$\kappa^{}_3$ induces the inner
automorphism $ \tilde{\kappa}^{}_3 $ of order three on $ E_8 $: $
\tilde{\kappa}^{}_3(\alpha)=\kappa
_3\alpha {\kappa^{}_3}^{-1}, \alpha \in E_8 $ and $ \kappa^{}_4$ induces the  inner automorphism $ \tilde{\kappa}^{}_4 $ of order four on  $ E_8 $: $ \tilde{\kappa}^{}_4(\alpha)=\kappa^{}_4\alpha {\kappa^{}_4}^{-1}, \alpha \in E_8 $, so are on $ {E_8}^C $.

\vspace{2mm}

Now, we will study the subgroup $({E_8}^C)^{\kappa^{}_4}$ of $ {E_8}^C $:
\begin{align*}
({E_8}^C)^{\kappa^{}_4}=\left\lbrace \alpha \in {E_8}^C \relmiddle{|}
\kappa^{}_4 \alpha=\alpha \kappa^{}_4 \right\rbrace .
\end{align*}


We prove the following lemma needed in the proof of theorem below.

\begin{lem}\label{lem 8.1}
The Lie algebra $({\mathfrak{e}_8}^C)^{\kappa^{}_4}$ of
the group $({E_8}^C)^{\kappa^{}_4}$ coincides with the Lie algebra $({\mathfrak{e}_8}^C)^{\kappa^{}_3}$ of the group $({E_8}^C)^{\kappa^{}_3}${\rm :}$ ({\mathfrak{e}_8}^C)^{\kappa^{}_4}=({\mathfrak{e}_8}^C)^{\kappa^{}_3} $.
\end{lem}
\begin{proof}
By an argument similar to the proof of $ \mathfrak{g}_0=(\mathfrak{g})^{\gamma_3} $ in the beginning of this section,
$ \mathfrak{g}_0=(\mathfrak{g})^{\gamma_4}:=\left\lbrace X \in \mathfrak{g} \relmiddle{|} \tilde{\gamma}_4 X=X \right\rbrace $ is proved, where $ \tilde{\gamma}_4:=\exp\left(\dfrac{2\pi i}{4}\ad\, Z \right)  $. Hence, by replacing $ \mathfrak{g} $ and $ \gamma_3,\gamma_4 $ with $ {\mathfrak{e}_8}^C $ and $ \kappa_3,\kappa_4 $, respectively, we obtain
\begin{align*}
({\mathfrak{e}_8}^C)^{\kappa_4}=\mathfrak{g}_0=({\mathfrak{e}_8}^C)^{\kappa_3},
\end{align*}
where the Lie algebra $ \mathfrak{g}_0 $ above is the same one obtained in \cite[Theorem 5.7]{iy0}.
\end{proof}

\vspace{1mm}

Now, we determine the structure of the group $({E_8}^C)^{\kappa^{}_4}$.

\begin{thm}\label{thm 8.2}
    The group $({E_8}^C)^{\kappa^{}_4}$ is isomorphic to the
    group $(C^* \times Spin(14,C))/\Z_4,\allowbreak \Z_4 =\{(1,1), (-1, \phi(-1)), (i, \phi(-i)), (-i, \phi(i))\}${\rm :} $({E_8}^C)^{\kappa^{}_4} \cong (C^* \times Spin(14,C))/\Z_4$.
\end{thm}
\begin{proof}
    First, we will prove that the group $({E_8}^C)^{\kappa^{}_4}$ coincides with the group $ ({E_8}^C)^{\kappa^{}_3} $. Since $ {E_8}^C $ is the simply connected complex Lie group type $ E_8 $, both of the groups  $ ({E_8}^C)^{\kappa^{}_4} $ and $ ({E_8}^C)^{\kappa^{}_3} $ are connected ((\cite[Preliminaries Lemma 2.2]{miya2}) in \cite{ra}),
    moreover together with $ ({\mathfrak{e}_8}^C)^{\kappa^{}_4}=({\mathfrak{e}_8}^C)^{\kappa^{}_3} $ (Lemma \ref{lem 8.1}), we have
    $ ({E_8}^C)^{\kappa^{}_4}=({E_8}^C)^{\kappa^{}_3} $.

    Here, we define a mapping  $ \varphi:C^* \times Spin(14,C) \to ({E_8}^C)^{\kappa^{}_4} $ by
    the same mapping defined in the proof of \cite[Theorem 5.10]{miya0} as follows:
    \begin{align*}
        \varphi(a,\beta)=\phi(a)\beta,
    \end{align*}
    where $ \phi $ is defined in \cite[Subsection 5.3 (p.45)]{miya0} and $ Spin(14,C) $ is constructed in \cite[Proposition 5.8.7]{miya0}.

     Therefore, from $ ({E_8}^C)^{\kappa^{}_4}=({E_8}^C)^{\kappa^{}_3} $, we have the required isomorphism
    \begin{align*}
        ({E_8}^C)^{\kappa^{}_4} \cong (C^* \times Spin(14,C))/\Z_4.
    \end{align*}
\end{proof}

\subsection{The group $(E_8)^{\kappa_4}$ }

In this subsection, as for the construction of $ Spin(14,C) $, we will give a briefly explain based on \cite[Section 5.3]{miya0}, and note that several signs in \cite[Section 5.3]{miya0} are changed as mentioned in the beginning of this section. After that, we will construct the group $ Spin(14) $ in $ E_8 $ and determine the structure of the group $(E_8)^{\kappa_4}$.

Besides, as for the embedding sequence of the exceptional compact Lie groups: $ Spin(8) \subset F_4 \subset E_6 \subset E_7 \subset E_8  $, see \cite[Theorems 2.7.1, 3.7.1, 4.7.2, 5.7.3 ]{iy0} in detail.



\vspace{1mm}

We define $14$-dimensional $C$-vector subspaces
$({\mathfrak{e}_8}^C)_2, ({\mathfrak{e}_8}^C)_{-2}$ of
${\mathfrak{e}_8}^C$ by
\begin{align*}
({\mathfrak{e}_8}^C)_2
&=\left\lbrace R \in
{\mathfrak{e}_8}^C \,|\,(\ad\kappa) R=2R \right\rbrace
\\
&=\left\lbrace  R=(\varPhi, 0, Q, 0,0, t) \relmiddle{|}
\begin{array}{l}\varPhi=\varPhi(0,0,\varrho_1 E_1,0), \varrho_1 \in C,\\
Q=(\xi_2 E_2 + \xi_3 E_3 + F_1(x_1), \eta_1 E_1, 0,
\eta),\\
\qquad\; \xi_k, \eta_1, \eta \in C, x_1 \in
\mathfrak{C}^C,\\
t \in C
\end{array}
\right\rbrace ,
\\[2mm]
({\mathfrak{e}_8}^C)_{-2}
&=\left\lbrace R \in {\mathfrak{e}
    _8}^C \relmiddle{|}(\ad\kappa) R=-2R \right\rbrace
\\
&=\left\lbrace  R=(\varPhi, P, 0, 0,s, 0) \relmiddle{|}
\begin{array}{l}\varPhi=\varPhi(0,\upsilon_1 E_1,0,0),\upsilon_1 \in C,\\
P=(\xi_1 E_1, \eta_2 E_2 + \eta_3 E_3 + F_1(y_1), \xi, 0),
\\
\qquad\; \xi_1, \eta_k, \xi \in C, y_1 \in
\mathfrak{C}
^C,\\
s \in C
\end{array}
\right\rbrace,
\end{align*}
respectively, where $\kappa=(\varPhi(-2E_1 \vee E_1, 0,0, -1), 0,0,-1,0,0)
\in {\mathfrak{e}_8}^C$ is used in previous subsection .

We define a $C$-linear mapping
$\zeta : {\mathfrak{e}_8}^C \to {\mathfrak{e}_8}^C$
by
\begin{align*}
\zeta({\varPhi}, P, Q, r, s, t) =
(\zeta_1{\varPhi}{\zeta_1}^{-1}, i\zeta_1Q, i\zeta_1P, -r, t,
s),
\end{align*}
where the $C$-linear transformation $\zeta_1 $ of $ \mathfrak{P}^C $ on the right hand side is defined by
\begin{align*}
\zeta_1(X, Y, \xi, \eta) =
(\begin{pmatrix}i\eta & x_3 & \overline{x}_2 \\
\overline{x}_3 & i\eta_3 & - iy_1 \\
x_2 & - i\,\overline{y}_1 & i\eta_2
\end{pmatrix},
\begin{pmatrix}i\xi & y_3 & \overline{y}_2 \\
\overline{y}_3 & i\xi_3 & - ix_1 \\
y_2 & - i\,\overline{x}_1 & i\xi_2
\end{pmatrix},
i\eta_1, i\xi_1 ),
\end{align*}
\noindent In particular, the restriction of the mapping $ \zeta $ to $ ({\mathfrak{e}_8}^C)_{-2} $ induces a mapping $ ({\mathfrak{e}_8}^C)_{-2} \to ({\mathfrak{e}_8}^C)_2 $, so the explicit form of its restriction mapping
is given by
\begin{align*}
&\quad \zeta(\varPhi(0,\upsilon_1 E_1,0,0),(\xi_1 E_1,
\eta_2
E_2 + \eta_3 E_3 + F_1(y_1), \xi, 0),0,0,s,0 )
\\
&=(\varPhi(0,0,\upsilon_1 E_1,0),0,(-\eta_3 E_2 -\eta_2 E_3 +
F_1(y_1), -\xi E_1, 0, -\xi_1),0,0,s).
\end{align*}
Note that the restriction of the mapping $ \zeta $ to $ ({\mathfrak{e}_8}^C)_2 $ is also denoted by the same sign.

Moreover, we define a $C$-linear mapping $\delta :
({\mathfrak{e}_8}^C)_2 \to ({\mathfrak{e}_8}^C)_2$
\begin{align*}
\delta(\varPhi(0,0,\varrho_1 E_1,0),0,Q, 0,0,t)
=(\varPhi(0,0,-t E_1,0),0,Q, 0,0,-\varrho_1).
\end{align*}
We denote the composition mapping $\delta\zeta :
({\mathfrak{e}_8}^C)_{-2} \to ({\mathfrak{e}_8}^C)_2$ of
$\zeta$ and $\delta$ by $\zeta_{\delta}$, then the explicit form of the mapping $ \zeta_{\delta} $ is given by
\begin{align*}
&\quad \zeta_{\delta}(\varPhi(0,\upsilon_1 E_1,0,0),(\xi_1 E_1,
\eta_2 E_2 + \eta_3 E_3 + F_1(y_1), \xi, 0),0,0,s,0 )
\\
&=(\varPhi(0,0,-s E_1,0),0,(-\eta_3 E_2 -\eta_2 E_3 + F_1(y_1),
-\xi E_1, 0, -\xi_1),0,0,-\upsilon_1),
\end{align*}
in addition, the explicit form of the inverse mapping ${\zeta_\delta}^{-1}:({\mathfrak{e}_8}^C)_2 \to ({\mathfrak{e}_8}^C)_{-2}$ is given by
\begin{align*}
&\quad {\zeta_\delta}^{-1}(\varPhi(0,0,\varrho_1 E_1, 0),0,
(\xi_2
E_2+\xi_3 E_3+F_1(x_1), \eta_1 E_1, 0,\eta), 0,0,t)
\\
&=(\varPhi(0,-t E_1,0,0),(-\eta E_1, -\xi_3 E_2-\xi_2
E_3+F_1(x_1),-\eta_1,0),0,0,-\varrho_1,0).
\end{align*}

Now, we define a subgroup $ (G_{14})^C $ of $ {E_8}^C $ by
\begin{align*}
(G_{14})^C:=\left\lbrace \beta \in {E_8}^C \relmiddle{|} (\ad
\kappa)\beta=\beta(\ad \kappa), \zeta_\delta \beta R=\beta
\zeta_\delta R, R \in ({\mathfrak{e}_8}^C)_{-2} \right\rbrace.
\end{align*}
Then, from \cite[Proposition 5.8.7]{miya0}, the group $ (G_{14})^C $ is isomorphic to the group $ Spin(14,C) $ as the universal covering group of $ SO(14,C) $:
\begin{align*}
   (G_{14})^C \cong Spin(14,C).
\end{align*}
In particular, note that $ (G_{14})^C, ({\mathfrak{e}_8}^C)_{-2} $ above are denoted by $ G_{14}, (V^C)^{14} $ in \cite[Proposition 5.8.7]{miya0}, respectively.

We prove the following lemma needed below and later.
\begin{lem}\label{lem 8.3}
    The $C$-linear transformation $\tau\lambda_\omega$ satisfies the formula $ (\ad\kappa)\tau\lambda_\omega=-\tau\lambda_\omega\allowbreak (\ad\kappa) $ and
    commutes with the $C$-linear transformation $ \kappa^{}_4${\rm :}$(\tau\lambda_\omega)\kappa^{}_4=\kappa^{}_4(\tau\lambda_\omega)
    $, where $\tau\lambda_\omega$ is a composition transformation
    of $\tau$ and $\lambda_\omega$ defined in
    \cite[Preliminaries (p.96)] {miya2}.
\end{lem}
\begin{proof}
    First, we denote $ \varPhi(-2E_1 \vee E_1, 0,0, -1) $ by $ \varPhi_\kappa $ only in this lemma: $ \varPhi_\kappa:=\varPhi(-2E_1 \vee E_1, 0,0, -1) $.
    Then, note that $ \tau\lambda\varPhi_\kappa=-\varPhi_\kappa\tau\lambda $, it follows that
    \begin{align*}
        &\quad(\ad\kappa)\tau\lambda_\omega(\varPhi,P,Q,r,s,t)
        \\
        &=\ad\kappa(\tau\lambda\varPhi\lambda^{-1}\tau,\tau\lambda Q,-\tau\lambda P, -\tau r, -\tau t,-\tau s)
        \\
        &=([\varPhi_\kappa,\tau\lambda\varPhi\lambda^{-1}\tau],\varPhi_\kappa(\tau\lambda Q)-\tau\lambda Q,-\varPhi_\kappa(\tau\lambda P)-\tau\lambda P,0,2\tau t,-2\tau s)
        \\
        &=-(\tau\lambda[\varPhi_\kappa,\varPhi]\lambda^{-1}\tau,\tau\lambda(\varPhi_\kappa Q+Q),-\tau\lambda(\varPhi_\kappa P-P),0,-2(\tau t), 2(\tau s))
        \\
        &=-\tau\lambda_\omega(\ad\kappa)(\varPhi,P,Q,r,s,t), \,\,(\varPhi,P,Q,r,s,t) \in \mathfrak{e}_8^C,
    \end{align*}
    that is, $ (\ad\kappa)\tau\lambda_\omega=-\tau\lambda_\omega(\ad\kappa) $. The first half is proved.

    The second half is easily proved by doing straightforward computation under the definitions of $ \tau\lambda_\omega $ and $ \kappa^{}_4 $.
\end{proof}

Here, in order to prove the proposition below , we use the following lemma.

\begin{lem}\label{lem 8.4}
   {\rm (1)} For $ R \in ({\mathfrak{e}_8}^C)_{-2} $, the formula  $(\tau\lambda_\omega) \zeta_\delta R={\zeta_\delta}^{-1} (\tau\lambda_\omega) R $ holds.
\vspace{1mm}

  {\rm (2)} For $\beta \in Spin(14,C)$, $\beta$ satisfies the formula ${\zeta_\delta}^{-1}\beta R'=\beta{\zeta_\delta}^{-1} R', R' \in
  ({\mathfrak{e}_8}^C)_2 $.
\end{lem}
\begin{proof}
  (1) Under the definition of $ \tau\lambda_\omega $ and the mappings $ \zeta_{\delta} ,{\zeta_{\delta}}^{-1}$ mentioned above, we do straightforward computation of both sides:
\begin{align*}
&\quad (\tau\lambda_\omega) \zeta_\delta R
\\
&=(\tau\lambda_\omega)\zeta_\delta(\varPhi(0,\upsilon_1E_1,0,0),(\xi_1E_1,\eta_2E_2+\eta_3E_3+F_1(y_1),\xi,0),0,0,s,0)
\\
&=\tau\lambda_\omega(\varPhi(0,0,-sE_1,0),0,(-\eta_3E_2-\eta_2E_3+F_1(y_1),-\xi E_1,0,-\xi_1),0,0,-\upsilon_1)
\\
&=(\tau\lambda\varPhi(0,0,-sE_1,0)\lambda^{-1}\tau,\tau\lambda(-\eta_3E_2-\eta_2E_3+F_1(y_1),-\xi E_1,0,-\xi_1),0,0,\tau\upsilon_1,0)
\\
&=(\varPhi(0,\tau sE_1,0,0),(-\tau\xi E_1,\tau\eta_3E_2+\tau\eta_2E_3-F_1(\tau y_1),-\tau\xi_1,0),0,0,\tau\upsilon_1,0),
\\[1mm]
&\quad {\zeta_\delta}^{-1} (\tau\lambda_\omega) R
\\
&={\zeta_\delta}^{-1}(\tau\lambda\varPhi(0,\upsilon_1E_1,0,0)\lambda^{-1}\tau,0,-\tau\lambda(\xi_1E_1,\eta_2E_2+\eta_3E_3+F_1(y_1),\xi,0),0,0,
-\tau s)
\\
&={\zeta_\delta}^{-1}(\varPhi(0,0,-\tau\upsilon_1E_1,0),0,(-\tau\eta_2E_2-\tau\eta_3E_3-F_1(\tau y_1),\tau\xi_1E_1,0,\tau\xi),0,0,-\tau s)
\\
&=(\varPhi(0,\tau sE_1,0,0),(-\tau\xi E_1,\tau\eta_3E_2+\tau\eta_2E_3-F_1(\tau y_1),-\tau\xi_1,0),0,0,\tau\upsilon_1,0).
\end{align*}
With above, the required formula is proved.

\vspace{2mm}
  (2) Let $ Spin(14,C) $ as the group $ (G_{14})^C $. In the formula $ \zeta_\delta \beta R=\beta\zeta_\delta R,R \in ({\mathfrak{e}_8}^C)_{-2} $, the required formula is proved by setting $ \zeta_{\delta} R=R', R' \in ({\mathfrak{e}_8}^C)_2 $.

\end{proof}


\begin{prop}\label{prop 8.5}
    The $C$-linear transformation $\tau\lambda_\omega$ induces
    the involutive inner automorphism of the group
    $Spin(14,C)${\rm :} $\tilde{\tau\lambda}_\omega(\beta)=(\tau\lambda_\omega)
     \beta (\lambda_\omega \tau), \beta \in Spin(14,C)$.
\end{prop}
\begin{proof}
    Let $Spin(14,C)$ as the group
    $(G_{14})^C$.
     We define a mapping $ g: (G_{14})^C \to  (G_{14})^C $ by
\begin{align*}
    g(\beta)=(\tau\lambda_\omega) \beta (\lambda_\omega \tau).
\end{align*}

    We will prove $g(\beta) \in (G_{14})^C$. First, it follows from the first half of Lemma \ref{lem 8.3} that
    \begin{align*}
        (\ad\kappa)g(\beta)&=(\ad\kappa)(\tau\lambda_\omega) \beta (\lambda_\omega \tau)
        \\
        &=-(\tau\lambda_\omega) (\ad\kappa)\beta (\lambda_\omega \tau)
        \\
        &=-(\tau\lambda_\omega) \beta(\ad\kappa) (\lambda_\omega \tau)
       \\
       &=-(\tau\lambda_\omega) \beta(-\lambda_\omega \tau (\ad\kappa))
        \\
        &=(\tau\lambda_\omega) \beta(\lambda_\omega \tau) (\ad\kappa)
        \\
        &=g(\beta)(\ad\kappa),
    \end{align*}
    that is, $ (\ad\kappa)g(\beta)=g(\beta)(\ad\kappa) $.

    Next, as for
    ${\zeta_\delta} g(\beta) R=g(\beta){\zeta_\delta}
    R, R \in ({\mathfrak{e}_8}^C)_{-2} $, it follows from Lemma \ref{lem 8.4} (1), (2) that
    \begin{align*}
    {\zeta_\delta} g(\beta) R
    &={\zeta_\delta}(\tau\lambda_\omega \beta
    \lambda_\omega \tau)R \;\;\;\;\;((\tau\lambda_\omega \beta
    \lambda_\omega \tau)R \in ({\mathfrak{e}_8}^C)_{-2} )
    \\
    &=(\tau\lambda_\omega){\zeta_\delta}^{-1}
    (\tau\lambda_\omega) (\tau\lambda_\omega \beta \lambda_\omega
    \tau R)
   \\
   &=\tau\lambda_\omega{\zeta_\delta}^{-1}  \beta
    \lambda_\omega \tau R \;\;\;\;\;(\lambda_\omega
    \tau R \in ({\mathfrak{e}_8})_2)
   \\
    &=\tau\lambda_\omega \beta
    {\zeta_\delta}^{-1}\lambda_\omega \tau
    R
    \\
    &=\tau\lambda_\omega \beta \tau \lambda_\omega
    \zeta_\delta R
    \\
    &=(\tau\lambda_\omega \beta  \lambda_\omega \tau)
    \zeta_\delta R
    \\
    &=g(\beta)\zeta_\delta R,
    \end{align*}
    that is, $ {\zeta_\delta} g(\beta) R=g(\beta)\zeta_\delta R, R \in ({\mathfrak{e}_8}^C)_{-2} $.

    Hence  we have $ g(\beta) \in (G_{14})^C $. The proof of this proposition is completed.
\end{proof}

From Proposition \ref{prop 8.5}, we can define a subgroup $
(Spin(14, C))^{\tau\lambda_\omega}$ of $
Spin(14,C) $:
\begin{align*}
(Spin(14, C))^{\tau\lambda_\omega}=\left\lbrace \beta \in
Spin(14,C) \relmiddle{|}
(\tau\lambda_\omega) \beta=\beta(\tau\lambda_\omega) \right
\rbrace .
\end{align*}

We prove the following lemma needed in the proof of theorem below .

\begin{lem}\label{lem 8.6}
    The Lie algebra $(\mathfrak{spin}(14, C))^{\tau\lambda_
    \omega}$
    of the group $(Spin(14, C))^{\tau\lambda_\omega}$ is given
    by
    \begin{align*}
    &\quad (\mathfrak{spin}(14, C))^{\tau\lambda_\omega}=\left\lbrace R \in
    (\mathfrak{g}_{14})^C \relmiddle{|}(\tau\lambda_\omega) R=R \right\rbrace
     \\
    &=\left\{R=(\varPhi, P, -\tau\lambda P, r, 0, 0) \left|\!\!
    \begin{array}{l}
    \,\,\varPhi=\varPhi(D+\tilde{A}(d_1)
    \\
    \qquad +i(\tau_1 E_1+\tau_2 E_2
     +\tau_3 E_3+F_1(t_1))^\sim,
     A, -\tau A,\nu),
    \\
    \qquad\quad D \in \mathfrak{so}(8), d_1, t_1 \in \mathfrak{C},
    \\
    \qquad\quad \tau_k \in \R, \tau_1+\tau_2+\tau_3=0,
    \\
    \qquad\quad A=\alpha_2 E_2+\alpha_3E_3+F_1(a_1) \in \mathfrak{J}^C,
    \nu \in i\R,
    \vspace{1mm}\\
    \,\,P=(\rho_2 E_2+\rho_3 E_3 +F_1(p_1),\rho_1 E_1,0,\rho) \in \mathfrak{P}^C,
    \vspace{1mm}\\
    \,\,r \in i\R,
    \vspace{1mm}\\
    \,\, i\tau_1+(2/3)\nu+2r=0
    \end{array}
    \right. \right \}.
    \end{align*}
    where the Lie algebra $ (\mathfrak{g}_{14})^C $ is defined
    as $ \mathfrak{g}_{14} $ in \cite[Lemma 5.8.3]{miya0}.

    In particular, we have $\dim((\mathfrak{spin}(14,C)^{\tau\lambda_
    \omega})=(28+16+2+20+1)+24+1-1=91$.
\end{lem}
\begin{proof}
    Let $ R=(\varPhi(\phi, A,B,\nu),P,Q,r,0,0) \in (\mathfrak{g}_{14})^C $, where
    \begin{align*}
    &\phi=D+\tilde{A}_1(d_1)+(\tau_1E_1+\tau_2E_2+\tau_3E_3+F_1(t_1))^\sim, D \in \mathfrak{so}(8,C),
    \\
    &\hspace{45mm} d_1,t_1 \in \mathfrak{C}^C, \tau_k \in C,\tau_1+\tau_2+\tau_3=0,
    \\
    &A=\alpha_2E_2+\alpha_3E_3+F_1(a_1) \in \mathfrak{J}^C, \alpha_k \in C, a_1 \in \mathfrak{C}^C,
    \\
    &B=\beta_2E_2+\beta_3E_3+F_1(b_1) \in \mathfrak{J}^C, \beta_k \in C, b_1 \in \mathfrak{C}^C,
    \\
    &\nu \in C,
    \\
    &P=(\rho_2 E_2+\rho_3 E_3 +F_1(p_1),\rho_1 E_1,0,\rho) \in \mathfrak{P}^C,\rho_k,\rho \in C, p_1 \in \mathfrak{C}^C,
    \\
    &Q=(\zeta_1E_1,\zeta_2 E_2+\zeta_3 E_3 +F_1(z_1),\zeta_1 E_1,\zeta,0) \in \mathfrak{P}^C,\zeta_k,\zeta \in C, z_1 \in \mathfrak{C}^C,
    \\
    &r \in C,
    \\
    &\tau_1+(2/3)\nu+2r=0.
    \end{align*}
  Then it follows from
   \begin{align*}
     (\tau\lambda_\omega)R&=(\tau\lambda_\omega)(\varPhi(\phi, A,B,\nu),P,Q,r,0,0,)
     \\
     &=(\tau\lambda\varPhi(\phi, A,B,\nu)\lambda^{-1}\tau,\tau\lambda Q,-\tau\lambda P,-\tau r,0,0)
   \end{align*}
   that
   \begin{align*}
     \tau\lambda\varPhi(\phi, A,B,\nu)\lambda^{-1}\tau=\varPhi(\phi, A,B,\nu),\;\;\tau\lambda Q=P,\;\;-\tau\lambda P=Q,\;\;-\tau r=r.
   \end{align*}

   Hence, using the formula $ \tau\lambda\varPhi(\phi, A,B,\nu)\lambda^{-1}\tau=\varPhi(-\tau\,{}^t\phi\tau,-\tau B,-\tau A,-\tau \nu) $, the required result is obtained.
\end{proof}

  From Lemma \ref{lem 8.6}, we have
 \begin{align*}
   (\mathfrak{spin}(14, C))^{\tau\lambda_\omega} \subset \mathfrak{e}_8=\left\lbrace R=(\varPhi,P,-\tau\lambda P,r,s,-\tau s) \relmiddle{|}
   \begin{array}{l}
    \varPhi \in \mathfrak{e}_7,P \in \mathfrak{P}^C,
    \\
    r \in i\R,s \in C
   \end{array} \right\rbrace .
 \end{align*}

Now, we construct the spinor group $ Spin(14) $ in $ E_8 $.

\begin{thm}\label{thm 8.8}
    The group $(Spin(14, C))^{\tau\lambda_\omega}$ is
    isomorphic to the group $Spin(14)$
    {\rm :} \\
   $(Spin(14,C))^{\tau\lambda_\omega} \cong Spin(14)$.
\end{thm}
\begin{proof}
    First, since the group $ Spin(14,C) $ is the simply connected Lie group, the group $ (Spin(14, C))^{\tau\lambda_\omega} $ is connected ((\cite[Preliminaries Lemma 2.2]{miya2}) in \cite{ra}). Hence, since both of the groups $ (Spin(14, C))^{\tau\lambda_\omega} $ and $ E_8 $ are connected and $ (\mathfrak{spin}(14, C))^{\tau\lambda_\omega} \subset \mathfrak{e}_8 $, we confirm $ (Spin(14, C))^{\tau\lambda_\omega} \subset E_8 $.

    We define a $14$-dimensional $\R$-vector space
    $V^{14}$ by
    \begin{align*}
    V^{14}&=\left\lbrace  R \in ({\mathfrak{e}_8}^C)_{-2} \relmiddle{|}
    (\tau\lambda_\omega)\zeta_\delta R=-R  \right\rbrace
    \\
    &\!=\left\lbrace \! R=(\varPhi(0, \upsilon E_1,0,0), (\xi E_1,
    \eta E_2-\tau\eta E_3+F_1(y), \tau \xi,0),0,0,-\tau \upsilon,0)\!
    \relmiddle{|}\!\!\!
    \begin{array}{l}
    \upsilon \in C,
    \\
    \xi \in C,
    \\
    \eta \in C,
    \\
    y \in \mathfrak{C}
    \end{array} \!\!\right\rbrace
    \end{align*}
 with the norm
 \begin{align*}
   (R, R)_\zeta=\dfrac{1}{30}B_8(\zeta_\delta R,R)=4(\tau \upsilon)\upsilon+(\tau \eta)\eta+y\overline{y}+(\tau \xi)\xi,
 \end{align*}
 where $ B_8 $ is the
 Killing form of $ {\mathfrak{e}_8}^C $ (as for the Killing form $ B_8 $, see \cite[Theorem 5.3.2]{iy0} in detail).
 Obviously, the group $(Spin(14, C))^{\tau\lambda_\omega}$ acts on $V^{14}$.

 Here, let the orthogonal group
 \begin{align*}
    O(14)=O(V^{14})=\left\lbrace \alpha \in \Iso_{\bm{R}}(V^{14})\relmiddle{|}(\alpha R, \alpha
    R)_\zeta=(R,R)_\zeta \right\rbrace.
 \end{align*}
 We consider the restriction $ \alpha\bigm|_{V^{14}} $ of $ \alpha \in (Spin(14, C))^{\tau\lambda_\omega} $ to $ V^{14} $, then we see $ \alpha\bigm|_{V^{14}} \in O(14)=O(V^{14}) $. Indeed, the group $(Spin(14, C))^{\tau\lambda_\omega}$ acts on $V^{14}$, so that we have $  \alpha\bigm|_{V^{14}} \in \Iso_{\bm{R}}(V^{14}) $. Moreover, for $ R \in V^{14} $, it follows that
\begin{align*}
(\alpha\bigm|_{V^{14}}R, \alpha\bigm|_{V^{14}}R)_\zeta
&=\dfrac{1}{30}B_8(\zeta_{\delta}\alpha\bigm|_{V^{14}}R, \alpha\bigm|_{V^{14}}R)
\\
&=\dfrac{1}{30}B_8(\zeta_\delta\alpha R,\alpha R)
\\
&=\dfrac{1}{30}B_8(\alpha\zeta_{\delta}R, \alpha R)
\\
&=\dfrac{1}{30}B_8(\zeta_{\delta}R, R)
\\
&=(R,R)_\zeta.
\end{align*}
Hence we can define a homomorphism $ \pi: (Spin(14, C))^{\tau\lambda_\omega} \to O(14)=O(V^{14}) $ by
 \begin{align*}
 \pi(\alpha)=\alpha\bigm|_{V^{14}}.
 \end{align*}
 Moreover, since the mapping $ \pi $ is continuous and the group $ (Spin(14, C))^{\tau\lambda_\omega} $ is connected, $ \pi $ induces a homomorphism
 \begin{align*}
 \pi:(Spin(14,
 C))^{\tau\lambda_\omega} \to SO(14)=SO(V^{14}).
 \end{align*}

 We will determine $ \Ker\,\pi $. First, from the
 definition of kernel, we have
 \begin{align*}
    \Ker\,\pi&=\left\lbrace \beta \in  (Spin(14,
    C))^{\tau\lambda_\omega} \relmiddle{|} \pi(\beta)=1 \right\rbrace
    \\
    &=\left\lbrace \beta \in  (Spin(14, C))^{\tau\lambda_\omega}\relmiddle{|}
    \beta\bigm|_{V^{14}}=1 \right\rbrace (\subset E_8).
 \end{align*}
\vspace{-2mm}

    \noindent Then let $\beta \in {\rm Ker}\,\pi$. Since it follows
    from $ (\varPhi(0,-E_1,0,0), 0,0,0,1,0), i(\varPhi(0, E_1,0,0),\allowbreak 0,0,0,1,0) \in V^{14} $ that
 \begin{align*}
    \beta(\varPhi(0, -E_1,0,0), 0,0,0,1,0)&=(\varPhi(0,
    -E_1,0,0),0,0,0,1,0),
    \\
    \beta i(\varPhi(0,E_1,0,0),0,0,0,1,0)&=i(\varPhi(0,E_1,0,0),0,0,0,1,0),
 \end{align*}
    note that $ \beta \in \Iso_{{}_C}({\mathfrak{e}_8}^C) $, we have $\beta(0,0,0,0,1,0)=(0,0,0,0,1,0)$. Here, if $\beta(0,0,0,0,1,0)=(0,0,0,0,1,0)$, we have $\beta(0,0,0,0,0,1)=(0,0,0,0,0,1)$. Indeed, it follows from $ (\tau\lambda_\omega)\beta=\beta(\tau\lambda_\omega) $ that
\begin{align*}
(0,0,0,0,0,1)&=\tau\lambda_\omega(0,0,0,0,-1,0)=-\tau\lambda_\omega(0,0,0,0,1,0)
\\
&=-(\tau\lambda_\omega)\beta(0,0,0,0,1,0)=-\beta(\tau\lambda_\omega)(0,0,0,0,1,0)
\\
&=-\beta(0,0,0,0,0,-1)
\\
&=\beta(0,0,0,0,0,1),
\end{align*}
that is, $ \beta(0,0,0,0,0,1)=(0,0,0,0,0,1) $.
Hence we have
    $\beta \in E_7 \subset (E_8)_{(0, 0,0,0,0,1)}$ (\cite[Theorem
    5.7.3]{iy0}).

    \noindent Moreover, since it follows from $ (0,(E_1,0,1,0),0,0,0,0), i(0, (-E_1,0,1,0),0,0,0,0) \allowbreak \in V^{14} $ that
    \begin{align*}
        \beta(0,(E_1,0,1,0),0,0,0,0)&=(0,(E_1,0,1,0),0,0,0,0)
        \\
        \beta i(0,(-E_1,0,1,0),0,0,0,0)&=i(0,(-E_1,0,1,0),0,0,0,0),
    \end{align*}
     again note that $ \beta \in \Iso_{{}_C}({\mathfrak{e}_8}^C) $, we have $\beta(0,(0,0,1,0),0,0,0,0)=(0,(0,0,1,0),0,0,\allowbreak 0,0)$ and $ \beta(0,(E_1,0,0,0),0,0,0,0)=(0,(E_1,0,0,0),0,0,0,0) $, that is,
    $\beta(0,0,1,0)=(0,0,1,0)$ and $ \beta E_1=E_1 $ in $\mathfrak{P}^C$. Hence we have
    $\beta \in E_6 \subset (E_7)_{(0,0,1,0)}$ (\cite[Theorem
    4.7.2]{iy0}) and $ \beta $ satisfies $ \beta E_1=E_1 $.

    \noindent In addition, since $ (0,(0,E_2-E_3,0,0),0,0,0,0),(0,(0,i(E_2+E_3),0,0),0,0,0,0) \in V^{14} $,
    $ \beta $ satisfies $ \beta
    (E_2-E_3)=E_2-E_3$ and $ \beta(E_2+E_3)=E_2+E_3 $ in $\mathfrak{J}^C$ by an argument similar to above, together with $ \beta E_1=E_1 $ above, we have $\beta E_i=E_i,i=1,2,3 $, so we have $\beta \in F_4 \subset (E_6)_E$ (\cite[Theorem 3.7.1]{iy0}), and moreover we see
$\beta \in Spin(8) \subset (F_4)_{E_1,E_2,E_3}$
    (\cite[Theorem 2.7.1]{iy0}). Here, we denote $\beta \in
    Spin(8)$  by $\beta=(\beta_1, \beta_2, \beta_3) \in
    SO(8) \times SO(8) \times SO(8)$, then since $ \beta_1 $
    satisfies the condition $\beta_1 y=y $ for all $ y \in
    \mathfrak{C}$, we have $\beta_1=1$. Hence, from the
    Principle of triality on $SO(8)$, we have the following
    \begin{align*}
    \beta=(1,1,1)=1 \quad {\rm or}\quad \beta =(1,-1,-1)=:\sigma.
    \end{align*}
    Hence we have $ \Ker\,\pi \subset \{1,\sigma \} $ and vice versa, so that $ \Ker\,\pi=\{1, \sigma  \} \cong \Z_2$.

    Finally, since the group $SO(14)$ is connected
    and $\Ker\,\pi$ is discrete, together with $ \dim((\mathfrak{spin}
    (14,C)^{\tau\lambda_\omega})=91=\dim(\mathfrak{so}(14))$
    (Lemma \ref{lem 8.6}), $\pi$ is surjective.
    Thus we have the isomorphism $(Spin(14,C))^{\tau\lambda_\omega}/\Z_2 \cong SO(14)$.

    Therefore the group $(Spin(14, C))^{\tau\lambda_\omega}$ is
    isomorphic to the group $Spin(14)$ as the universal covering group of
    $SO(14)$:
    \begin{align*}
        (Spin(14,C))^{\tau\lambda_\omega} \cong Spin(14).
    \end{align*}
\end{proof}


Let the mapping $ \phi:C^* \to {E_8}^C $ defined in \cite[Subsection 5.3 (p.45)]{miya0}.

\noindent Then we prove the following lemma needed in the proof of theorem below.

\begin{lem}\label{lem 8.8}
	For $ a \in C^* $, the action of $\phi(a)$ on ${\mathfrak{e}_8}^C$ is given by
	$$
	\phi(a)(\varPhi, P, Q, r, s, t) =
	(\psi(a)\varPhi\psi(a)^{-1}, a\psi(a)P, a^{-1}\psi(a)Q, r,
	a^2s, a^{-2}t),
	$$
	where the action to $ \mathfrak{P}^C $ of $\psi(a) \in {E_7}^C$ on right hand side is defined by
	\begin{align*}
	\psi(a)(X, Y, \xi, \eta) =
	(\begin{pmatrix}a\xi_1 & x_3 & \overline{x}_2 \\
	\overline{x}_3 & a^{-1}\xi_2 & a^{-1}x_1 \\
	x_2 & a^{-1}\overline{x}_1 & a^{-1}\xi_3
	\end{pmatrix},
	\begin{pmatrix}a^{-1}\eta_1 & y_3 & \overline{y}_2 \\
	\overline{y}_3 & a\eta_2 & ay_1 \\
	y_2 & a\overline{y}_1 & a\eta_3
	\end{pmatrix}, a\xi, a^{-1}\eta ).
	\end{align*}
    Moreover, we have the following formula
    \begin{align*}
    (\tau\lambda_\omega)\phi(a)(\lambda_\omega\tau)=\phi((\tau a)^{-1}).
    \end{align*}
\end{lem}
\begin{proof}
	  As for the first half, its proof is proved in \cite[Lemma 5.9]{miya0}. As for the second half, we first have $ (\tau\lambda)\psi(a)(\lambda^{-1}\tau)= \psi((\tau a)^{-1}) $. Indeed, it follows that
    \begin{align*}
    &\quad (\tau\lambda)\psi(a)(\lambda^{-1}\tau)(X,Y,\xi,\eta)
    \\
    &=(\tau
    \lambda)\psi(a)(-\tau
    Y, \tau X, -\tau \eta, \tau \xi)
    \\
    &=(\tau\lambda) (\begin{pmatrix}
    -a\tau \eta_1 & -\tau y_3 & -\tau \overline{y}_2 \\
    -\tau \overline{y}_3 & -a^{-1}\tau \eta_2 & -a^{-1}\tau y_1
    \\
    -\tau y_2 &  -a^{-1}\tau \overline{y}_1  & -a^{-1}\tau
    \eta_3
    \end{pmatrix},
    \begin{pmatrix}
    a^{-1}\tau \xi_1 & \tau x_3 & \tau \overline{x}_2 \\
    \tau \overline{x}_3 & a \tau \xi_2 & a \tau x_1 \\
    \tau x_2 &  a \tau \overline{x}_1  & a \tau \xi_3
    \end{pmatrix},
    -a \tau\eta, a^{-1}\tau\xi)
    \\
    &=(\begin{pmatrix}
    (\tau a^{-1}) \xi_1 &  x_3 &  \overline{x}_2 \\
    \overline{x}_3 & (\tau a) \xi_2 & (\tau a) x_1 \\
    x_2 &  (\tau a) \overline{x}_1  & (\tau a) \xi_3
    \end{pmatrix},
    \begin{pmatrix}
    (\tau a) \eta_1 &  y_3 &  \overline{y}_2 \\
    \overline{y}_3 & (\tau a^{-1}) \eta_2 & (\tau a^{-1}) y_1 \\
    y_2 &  (\tau a^{-1}) \overline{y}_1  & (\tau a^{-1}) \eta_3
    \end{pmatrix},
    (\tau a^{-1})\xi, (\tau a) \eta )
    \\
    &=\psi((\tau a)^{-1}).
   \end{align*}
   Hence, using $ \tau
   \lambda_\omega(\varPhi, P, Q, r, s, t)=(\tau\lambda \varPhi
   \lambda^{-1}\tau, \allowbreak  \tau\lambda Q, -\tau\lambda P,
   -\tau r, -\tau t, -\tau s) $ and the formula $ (\tau\lambda)\psi(a)(\lambda^{-1}\tau)= \psi((\tau a)^{-1}) $ shown above, we can obtain the
   formula $ (\tau\lambda_\omega)\phi(a)(\lambda_\omega\tau)=\phi((\tau a)^{-1}) $ by doing straightforward computation.
\end{proof}

Now, we determine the structure of the group $
(E_8)^{\kappa_4}$.
\begin{thm}\label{thm 8.9}
    The group $(E_8)^{\kappa^{}_4}$ is isomorphic to the group $(U(1)
    \times Spin(14))/\Z_4, \Z_4\allowbreak =\{(1,1),\allowbreak (-1,
    \phi(-1)), (i, \phi(-i)), (-i, \phi(i))  \} ${\rm :} $
    (E_8)^{\kappa^{}_4} \cong (U(1) \times
    Spin(14))/\Z_4$.
\end{thm}
\begin{proof}
    Let $ U(1):=\{a \in C \,|\, (\tau a)a=1\} (\subset C^*) $ and $ Spin(14) (\subset Spin(14,C)) $
    as the group $ (Spin(14,C))^{\tau\lambda_\omega} $
 (Proposition \ref{thm 8.8}). Then we can define a mapping $
 \varphi_{\kappa^{}_4}: U(1) \times Spin(14) \to
 (E_8)^{\kappa^{}_4} $ by the restriction of the mapping $ \varphi:C^* \times Spin(14,C) \to ({E_8}^C)^{\kappa^{}_4} $:
  $ \varphi_{\kappa^{}_4}(a,\beta)=\varphi(a,\beta)=\phi(a)\beta $ (Theorem \ref{thm 8.2}).

 First, we will prove that $ \varphi_{\kappa^{}_4} $ is
 well-defined.
 It is clear $ \phi(a) \in ({E_8}^C)^{\kappa_4} $, so that
 we have $ \phi(a) \in (E_8)^{\kappa_4} $. Indeed, since $ a $ satisfies the condition $ (\tau a)a=1 $, we have $ ({\tau\lambda_\omega})\phi(a)(\lambda_\omega\tau)=
 \phi(a) $ by Lemma \ref{lem 8.8}, that is, $ \phi(a) \in (({E_8}^C)^{\kappa_4})^{\tau\lambda_\omega}$, moreover
 it follows from $ (\tau\lambda_\omega)\kappa^{}_4=\kappa^{}_4(\tau\lambda_\omega) $ (Lemma \ref{lem 8.3}) and $ ({E_8}
 ^C)^{\tau\lambda_\omega}=E_8 $ (\cite[Preliminaries (p.96)]{miya2}) that
 \begin{align*}
 \phi(a) \in (({E_8}^C)^{\kappa^{}_4})^{\tau\lambda_\omega}=(({E_8}
 ^C)^{\tau\lambda_\omega})^{\kappa^{}_4}=(E_8)^{\kappa_4}.
 \end{align*}
 By an argument similar to above, it follows from Lemma \ref{lem 8.3} that
\begin{align*}
\beta \in Spin(14)=(Spin(14,C))^{\tau\lambda_\omega} \subset (({E_8}^C)^{\kappa^{}_4})^{\tau\lambda_\omega}=(({E_8}^C)^{\tau\lambda_\omega})^{\kappa^{}_4}=(E_8)^{\kappa^{}_4},
\end{align*}
that is, $ \beta \in (E_8)^{\kappa^{}_4} $. Hence $ \varphi_{\kappa^{}_4} $ is well-defined. Subsequently, we will prove that $ \varphi_{\kappa^{}_4} $ is a homomorphism, however since $
 \varphi_{\kappa^{}_4} $ is the restriction of the mapping
 $ \varphi $, it is clear.

 Next, we will prove that $ \varphi_{\kappa^{}_4} $ is
 surjective.
 Let $\alpha \in (E_8)^{\kappa^{}_4} \subset
 ({E_8}^C)^{\kappa^{}_4}$, there exist  $a \in C^*$ and $\beta \in Spin(14,C)$ such that $ \alpha=\varphi(a, \beta)$ (Theorem \ref{thm 8.2}).
 Moreover, from the condition $(\tau\lambda_\omega)\alpha(\lambda_\omega \tau)=\alpha $, that is, $(\tau\lambda_\omega)\varphi(a,\beta)(\lambda_\omega \tau)=\varphi(a,\beta)$, we have $ \varphi((\tau a)^{-1},
 (\tau\lambda_\omega)\beta(\lambda_\omega
 \tau))=\varphi(a, \beta) $.
 Indeed, it follows from Lemma \ref{lem 8.8} that
  \begin{align*}
  (\tau\lambda_\omega)
  \varphi(a,\beta)(\lambda_\omega \tau)
  &=(\tau\lambda_\omega)
  \phi(a)\beta(\lambda_\omega \tau)
  \\
  &=(\tau\lambda_\omega)
  \phi(a)(\lambda_\omega \tau)(\tau\lambda_\omega)
  \beta(\lambda_\omega \tau)
  \\
  &=\phi((\tau a)^{-1})(\tau\lambda_\omega)
  \beta(\lambda_\omega \tau)
  \\
  &=\varphi((\tau a)^{-1},
 (\tau\lambda_\omega)
  \beta(\lambda_\omega \tau))
  \end{align*}
 that $ \varphi((\tau a)^{-1},
 (\tau\lambda_\omega)\beta(\lambda_\omega
 \tau))=\varphi(a, \beta) $.

  Hence, since $ \Ker\varphi=\{(1,1), (-1, \phi(-1)), (i, \phi(-i)), (-i, \phi(i))\} $, we have the following
  \begin{align*}
  \begin{array}{l}
  \mbox{(i)} \;\; \left\{\begin{array}{l}
  (\tau a)^{-1} = a
  \vspace{1mm}\\
  (\tau{\lambda_\omega})\beta({\lambda_\omega}\tau) = \beta,
  \end{array} \right.
  \qquad \qquad\;
  \mbox{(ii)} \;\; \left\{\begin{array}{l}
  (\tau a)^{-1} = -a
  \vspace{1mm}\\
  (\tau{\lambda_\omega})\beta({\lambda_\omega}\tau) =
  \phi(-1)\beta,
  \end{array} \right.
  \vspace{2mm}\\
  \hspace*{-1mm}
  \mbox{(iii)} \;\, \left\{\begin{array}{l}
  (\tau a)^{-1} = ia
  \vspace{1mm}\\
  (\tau{\lambda_\omega})\beta({\lambda_\omega}\tau) =
  \phi(-i)\beta,
  \end{array} \right.
  \quad\,\,\,\,\,
  \mbox{(iv)} \;\; \left\{\begin{array}{l}
  (\tau a)^{-1} = -ia
  \vspace{1mm}\\
  (\tau{\lambda_\omega})\beta({\lambda_\omega}\tau) =
  \phi(i)\beta.
  \end{array} \right.
  \end{array}
  \end{align*}

  Case (i). From $(\tau a)^{-1}=a$, we have $a \in U(1)=\{ a \in
  C^* \,|\,(\tau a)a=1 \}$. From $
  (\tau{\lambda_\omega})\beta({\lambda_\omega}\tau) = \beta$,
  we have $\beta \in Spin(14)$ (Theorem  \ref{thm 8.8}). Hence
  there exist $ a \in U(1) $ and $ \beta \in Spin(14) $ such
  that $ \alpha=\varphi(a,\beta)=\varphi_{\kappa^{}_4}(a,
  \beta)$.

  Case (ii). From $(\tau a)^{-1}=-a$, we have $(\tau a)a=-1$.
  However, this case is impossible because of $(\tau a)a >0$.

  Case (iii). From $(\tau a)^{-1}=ia$, we have $(\tau a)a=i$. As
  in Case (ii), this case is also impossible.

  Case (iv). From $(\tau a)^{-1}=-ia$, we have $(\tau a)a=-i$. As
  in Case (ii), this case is also impossible.

  \noindent With above, the proof of surjective is completed.

  Finally, we will determine $ \Ker\,\varphi_{\kappa^{}_4}$. Since $
  \varphi_{\kappa^{}_4} $ is the restriction of the
  mapping $ \varphi $, we have $ \Ker\,\varphi_{\kappa^{}_4}=\Ker\,\varphi $, that is, $ \Ker\,\varphi_{\kappa^{}_4}\!=\!\{(1,1), (-1,\phi(-1)), (i, \phi(-i)),\allowbreak (-i, \phi(i))  \} \cong \Z_4 $.

 Therefore we have the required isomorphism
 \begin{align*}
 (E_8)^{\kappa_4} \cong (U(1) \times Spin(14))/\Z_4.
 \end{align*}
\end{proof}


\section{Case 7. The automorphism $\tilde{\varepsilon}^{}_4$ of order four and the group $(E_8)^{{}_{\varepsilon^{}_4}}$}

We define a $ C $-linear transformations $ \varepsilon^{}_4$ of ${\mathfrak{e}_8}^C$ by
\begin{align*}
\varepsilon^{}_4(\varPhi, P, Q, r, s, t) &=
(\nu^{}_4\varPhi{\nu^{}_4}^{-1},
-\nu_4 P, -\nu_4 Q, r, s, t),
\end{align*}
where $\nu^{}_4 \in E_7$ on the right hand side is the same one as that
defined in previous section. Note that $ \varepsilon^{}_4 $ is the composition mapping of $ \nu^{}_4, \upsilon  \in E_7 \subset E_8 $, moreover since $ \nu^{}_4, \upsilon $ are expressed as elements of $ E_8 $ by
\begin{align*}
   \nu^{}_4&=\exp\left( \frac{2\pi i}{4}\ad(\varPhi(-2E_1 \vee E_1,0,0,-1),0,0,0,0,0)\right) ,
   \\
   \upsilon&=\exp\left( \frac{2\pi i}{4}\ad(\varPhi(0,0,0,6),0,0,0,0,0)\right) ,
\end{align*}
respectively and together with $ [\varPhi(-2E_1 \vee E_1,0,0,-1),\varPhi(0,0,0,6) ]=0 $, we have
\begin{align*}
    \varepsilon^{}_4=\exp\left( \frac{2\pi i}{4}\ad(\varPhi(-2E_1 \vee E_1,0,0,5),0,0,0,0,0)\right) .
\end{align*}
Hence it follows from above that $\varepsilon^{}_4 \in E_8 $ and
$(\varepsilon^{}_4)^4=1$, so that
$\varepsilon^{}_4$ induces the
inner automorphism $ \tilde{\varepsilon}^{}_4 $ of order four on $E_8$:
$\tilde{\varepsilon}^{}_4(\alpha)=\varepsilon^{}_4\alpha{\varepsilon^{}_4}^{-1},\alpha \in E_8$.
\vspace{1mm}

Now, we will study the subgroup $(E_8)^{\varepsilon^{}_4}$ of $ E_8 $:
$$
(E_8)^{\varepsilon^{}_4}=\left\lbrace \alpha \in E_8 \relmiddle{|}
\varepsilon^{}_4\alpha=\alpha\varepsilon^{}_4 \right\rbrace .
$$

The aim of this section is to determine the structure of the group $(E_8)^{\varepsilon^{}_4}$.
Before that, we make some preparations. First, in order to prove the proposition below, we use the following proposition and theorem.

\begin{prop}\label{prop 9.1}
   The group $ (E_7)^{\nu^{}_4} $ contains a group
   \begin{align*}
       \phi_{{}_{\scalebox{0.6}{$ U(1) $}}}(U(1))=\left\lbrace \phi_{{}_{\scalebox{0.6}{$ U(1) $}}}(\theta)\relmiddle{|} \theta \in U(1) \right\rbrace
   \end{align*}
which is isomorphic to the group $ U(1)=\left\lbrace \theta \in C \relmiddle{|} (\tau \theta)\theta=1 \right\rbrace  $, where $ \phi_{{}_{\scalebox{0.6}{$ U(1) $}}} $ is the restriction of the mapping $ \varphi_{{}_2} $ defined in \cite[Theorem 4.11.13]{iy0}. For $ \theta \in U(1) $, the mapping $ \phi_{{}_{\scalebox{0.6}{$ U(1) $}}}(\theta):\mathfrak{P}^C \to \mathfrak{P}^C $ is given by
\begin{align*}
    &\quad \phi_{{}_{\scalebox{0.6}{$ U(1) $}}}(\theta)(X,Y,\xi,\eta)
    \\
    &=\varphi^{}_2(\begin{pmatrix}
        \theta & 0 \\
        0  & \tau\theta
    \end{pmatrix})(
    \begin{pmatrix}
        \xi_1 & x_3 & \overline{x}_2 \\
        \overline{x}_3 & \xi_2 & x_1 \\
        x_2 & \overline{x}_1 & \xi_3
    \end{pmatrix},
    \begin{pmatrix}
        \eta_1 & y_3 & \overline{y}_2 \\
        \overline{y}_3 & \eta_2 & y_1 \\
        y_2 & \overline{y}_1 & \eta
    \end{pmatrix}, \xi, \eta ) \\
    & = ( \begin{pmatrix}
        \theta\xi_1 & x_3 & \overline{x}_2 \\
        \overline{x}_3 & (\tau\theta)\xi_2 & (\tau\theta)x_1 \\
        x_2 & (\tau\theta)\overline{x}_1 & (\tau\theta)\xi_3
    \end{pmatrix},
    \begin{pmatrix}
        (\tau\theta)\eta_1 & y_3 & \overline{y}_2 \\
        \overline{y}_3 & \theta\eta_2 & \theta y_1 \\
        y_2 & \theta\overline{y}_1 & \theta\eta_3
    \end{pmatrix}, \theta\xi, (\tau\theta)\eta ).
\end{align*}
\end{prop}
\begin{proof}
From the definition of $ \phi_{{}_{\scalebox{0.6}{$ U(1) $}}} $, it is clear $ \phi_{{}_{\scalebox{0.6}{$ U(1) $}}} \in E_7 $. Moreover, since $ \nu^{}_4 $ is expressed by $ \phi_{{}_{\scalebox{0.6}{$ U(1) $}}}(-i) $: $ \nu^{}_4=\phi_{{}_{\scalebox{0.6}{$ U(1) $}}}(-i) $, it is also clear $ \nu^{}_4\phi_{{}_{\scalebox{0.6}{$ U(1) $}}}(\theta)=\phi_{{}_{\scalebox{0.6}{$ U(1) $}}}(\theta)\nu^{}_4 $. Hence we have $ \phi_{{}_{\scalebox{0.6}{$ U(1) $}}}(\theta) \in (E_7)^{\nu^{}_4} $.
\end{proof}

\begin{thm}{\rm (\cite[Theorem 4.11.15]{iy0})}\label{thm 9.2}
    The group $ (E_7)^\sigma $ is isomorphic to the group $ (SU(2)\times Spin(12))/\Z_2,\Z_2=\{(E,1),(-E,-\sigma)\} ${\rm :} $ (E_7)^\sigma \cong (SU(2)\times Spin(12))/\Z_2 $.
\end{thm}
\begin{proof}
    We define a mapping $ \varphi:SU(2) \times Spin(12) \to (E_7)^\sigma $ by
    \begin{align*}
        \varphi(A,\beta)=\varphi_{{}_2}(A)\beta,
    \end{align*}
    where $ \varphi_{{}_2} $ is defined in \cite[Theorem 4.11.13]{iy0}.

    Then the mapping $ \varphi $ induces the required isomorphism. As for the $ \R $-linear transformation $ \sigma $ of $ \mathfrak{P}^C $, see \cite[Subsection 4.11(p.133)]{iy0} in detail.
\end{proof}

\begin{prop}\label{prop 9.3}
 The group $ (E_8)^{\varepsilon^{}_4} $ contains the group $ (E_7)^{\nu^{}_4} $ which is isomorphic to the group $(U(1) \times Spin(12))/\Z_2, \Z_2=\{(1, 1), (-1, -\sigma)
 \}${\rm :} $(E_8)^{\varepsilon^{}_4} \supset (E_7)^{\nu^{}_4} \cong (U(1) \times Spin(12))/
 \Z_2$.
\end{prop}
\begin{proof}
 Let $\alpha \in(E_7)^{\nu^{}_4}$. Note that $-1 \in z(E_7)$
 (the center of $E_7$),  we have
 \begin{align*}
     \varepsilon^{}_4\alpha(\varPhi, P, Q, r, s, t)
     &=\varepsilon^{}_4(\alpha\varPhi\alpha^{-1}, \alpha
     P, \alpha Q, r, s, t)
     \\
     &=(\nu^{}_4 \alpha\varPhi\alpha^{-1}{\nu^{}_4}^{-1},
     -\nu^{}_4 \alpha P, -\nu^{}_4 \alpha Q, r, s, t)
     \\
     &=(\alpha(\nu^{}_4 \varPhi{\nu^{}_4}^{-1})\alpha^{-1},
     \alpha(-\nu^{}_4  P), \alpha(-\nu^{}_4  Q), r, s, t)
     \\
     &=\alpha\varepsilon^{}_4(\varPhi, P, Q, r, s, t),\;\;(\varPhi, P, Q, r, s, t) \in {\mathfrak{e}_8}^C,
 \end{align*}
 that is, $\varepsilon^{}_4\alpha=\alpha\varepsilon^{}_4 $.
 Hence we have $ \alpha \in (E_8)^{\varepsilon^{}_4} $, so the first half is proved.

 Next, we will move the proof of the second half.
 Let $U(1) = \{ \theta \in C \, | \, (\tau\theta)\theta = 1
 \} $ and $ Spin(12) $ be constructed in \cite[Theorem 4.11.11]{iy0}.
  We define a mapping $\varphi_{{}_{\nu^{}_4}}: U(1) \times
 Spin(12) \to (E_7)^{\nu^{}_4}$ by
 \begin{align*}
  \varphi_{{}_{\nu^{}_4}}(\theta, \beta)=
  \phi_{{}_{U(1)}}(\theta)\beta.
 \end{align*}
Note that $ \varphi_{{}_{\nu^{}_4}} $ is the restriction of the mapping $ \varphi:SU(2)\times Spin(12) \to (E_7)^\sigma$ defined in the proof of Theorem \ref{thm 9.2}.

First, we will prove that $ \varphi_{{}_{\nu^{}_4}} $ is well-defined.
From Proposition \ref{prop 9.1}, we have $ \phi_{{}_{\scalebox{0.6}{$ U(1) $}}}(\theta) \in (E_7)^{\nu^{}_4} $. Since
$\varphi_2 (A) $ and $\beta \in Spin(12)$ are commutative, we have $ \nu^{}_4\beta=\beta\nu^{}_4 $ because of $\nu^{}_4= \phi_{{}_{\scalebox{0.6}{$ U(1) $}}}(-i)=\varphi_{{}_2}(\diag(-i,i)) $, that is, $ \beta \in (E_7)^{\nu^{}_4} $. Hence $ \varphi_{{}_{\nu^{}_4}} $ is well-defined. Subsequently, we will prove that $  \varphi_{{}_{\nu^{}_4}} $ is a homomorphism. However, since the mapping $  \varphi_{{}_{\nu^{}_4}} $ is the restriction of the mapping $ \varphi $, it is clear.

Next, we will prove that
 $ \varphi_{{}_{\nu^{}_4}}$ is surjective. Let $ \alpha \in (E_7)^{\nu^{}_4} $. Then it follows from
 $(\nu^{}_4)^2=(\phi_{{}_{U(1)}}(-i))^2=\phi_{{}_{U(1)}}(-1)=-
 \sigma$ that $ \alpha \in (E_7)^{-\sigma}=(E_7)^\sigma $. Hence there exist $A \in SU(2)$
 and $\beta \in Spin(12)$ such that $\alpha =\varphi(A,
 \beta)$ (Theorem \ref{thm 9.2}). Moreover from the condition $\nu^{}_4 \alpha {\nu^{}_4}^{-1} =\alpha $, that is, $\nu^{}_4 \varphi(A, \beta){\nu^{}_4}^{-1}= \varphi(A,
 \beta)$, using $ \nu^{}_4=\phi_{{}_{U(1)}}(-i) $ we have $\varphi(\begin{pmatrix}a & -b\\
     -c & d
 \end{pmatrix} , \nu^{}_4 \beta
 {\nu^{}_4}^{-1})=\varphi(\begin{pmatrix}a & b\\
     c & d
 \end{pmatrix} , \beta )$ as $A:=\begin{pmatrix}a & b\\
     c & d
 \end{pmatrix}$.
\vspace{1mm}

Thus, since $ \Ker\varphi=\{(E,1),(-E,-1)\} $, we have the following
 $$
 \left\{ \begin{array}{l}\begin{pmatrix}a & -b\\
         -c & d
     \end{pmatrix}=\begin{pmatrix}a & b\\
         c & d
     \end{pmatrix}
     \vspace{2mm}\\
     \nu^{}_4 \beta {\nu^{}_4}^{-1}=\beta
 \end{array} \right. \quad
 \text{or} \quad
 \left\{ \begin{array}{l}\begin{pmatrix}a & -b\\
         -c & d
     \end{pmatrix}=\begin{pmatrix}-a & -b\\
         -c & -d
     \end{pmatrix}
     \vspace{2mm}\\
     \nu^{}_4 \beta {\nu^{}_4}^{-1}=-\sigma\beta.
 \end{array} \right.
 $$
 In the latter case, from $\nu^{}_4 \beta=\beta \nu^{}_4$, we have $\beta=-\sigma\beta$. Hence
 the latter case is impossible. Indeed, if there exists $ \beta \in Spin(12) $ such that $ \beta=-\sigma\beta $. Then apply $ \beta^{-1} $ on both side of $ \beta=-\sigma\beta $, we have $ 1=-\sigma $. This is contradiction.

 \noindent In the former case. From the first condition $\begin{pmatrix}a & -b\\
     -c & d
 \end{pmatrix}=\begin{pmatrix}a & b\\
     c & d
 \end{pmatrix} \in SU(2)$, we have $A=\begin{pmatrix}a & 0\\
     0& \tau a
 \end{pmatrix}, (\tau a)a=1$, that is, $a \in U(1)$, and it is
 trivial that $\beta \in Spin(12)$. Thus there exist $ \theta \in U(1) $ and $ \beta \in Spin(12) $ such that $ \alpha=\varphi(\diag(\theta,\tau\theta),\beta)=\varphi_{{}_{\nu^{}_4}}(\theta,\beta) $.
 With above, the proof of surjective is completed.

 Finally, we will determine $ \Ker\,\varphi_{{}_{\nu^{}_4}} $. Since $ \varphi_{{}_{\nu^{}_4}} $ is restriction of the mapping $ \varphi $, we have $ \Ker\,\varphi_{{}_{\nu^{}_4}}=\Ker\,\varphi $, that is, $ \Ker\,\varphi_{{}_{\nu^{}_4}}= \{(1, 1), (-1, -\sigma) \}
 \cong \Z_2$.

 Therefore we have the required isomorphism
\begin{align*}
   (E_7)^{\nu^{}_4}
   \cong (U(1) \times Spin(12))/\Z_2
\end{align*}
\end{proof}

\begin{lem}\label{lem 9.4}
	{\rm (1)} The Lie algebra $ (\mathfrak{e}_8)^\upsilon $
	of the group $(E_8)^\upsilon $ is given by
	\begin{align*}
	(\mathfrak{e}_8)^\upsilon&=\left\lbrace \ad(R) \in \Der(\mathfrak{e}_8) \relmiddle{|} \upsilon \ad(R)=\ad(R)\upsilon \right\rbrace
\\
&\cong \left\lbrace R \in \mathfrak{e}_8 \relmiddle{|} \upsilon R=R \right\rbrace
	\\
	&=\left\lbrace R=(\varPhi, 0,0,r, s, -\tau s) \relmiddle{|}\varPhi \in \mathfrak{e}_7,r \in i\R, s \in C \right\rbrace.
	\end{align*}

	In particular, we have $ \dim((\mathfrak{e}_8)^\upsilon)=133+1+2=136 $.
\vspace{1mm}

    {\rm (2)} The Lie algebra $ (\mathfrak{e}_8)^{\varepsilon^{}_4} $
    of the group $(E_8)^{\varepsilon^{}_4}$ is given by
    \begin{align*}
    (\mathfrak{e}_8)^{\varepsilon^{}_4}&=\left\lbrace \ad(R) \in \Der(\mathfrak{e}_8) \relmiddle{|} \varepsilon^{}_4 \ad(R)=\ad(R)
    \varepsilon^{}_4 \right\rbrace
\\
&\cong \left\lbrace R \in
    \mathfrak{e}_8 \relmiddle{|} \varepsilon^{}_4 R=R \right\rbrace  \\
    &=\left\lbrace R=(\varPhi, 0,0,r, s, -\tau s) \relmiddle{|} \varPhi \in
    (\mathfrak{e}_7)^{\nu^{}_4} \cong \mathfrak{u}(1) \oplus
    \mathfrak{so}(12),r \in i\R, s \in C \right\rbrace .
    \end{align*}

    In particular, we have $ \dim((\mathfrak{e}_8)^{\varepsilon^{}_4})=(1+66)+1+2=70$.
\end{lem}
\begin{proof}
    By doing straightforward computation, we can prove this lemma. The  Lie-isomorphism $(\mathfrak{e}_7)^{\nu^{}_4}
    \cong \mathfrak{u}(1) \oplus \mathfrak{so}(12)$ follows from the group isomorphism $(E_7)^{\nu^{}_4} \cong (U(1) \times Spin(12))/\Z_2$
(Proposition \ref{prop 9.3}) .
\end{proof}

\begin{prop}\label{prop 9.5}
    The group $(E_8)^{\varepsilon^{}_4}$ contains a
    subgroup
    $$
    \phi_\upsilon(SU(2)) = \left\lbrace  \phi_\upsilon(A) \in E_8 \relmiddle{|}
    A \in SU(2) \right\rbrace
    $$
    which is isomorphic to the group $SU(2) = \{ A \in M(2, C)
    \, | \, (\tau\,{}^t\!A)A = E, \det A = 1 \}$,\vspace{1mm} where $ \phi_\upsilon $ is defined in Proposition {\rm \ref{prop 7.3}}.
\end{prop}
\begin{proof}
    For $A = \begin{pmatrix}
    a & -\tau b \\
    b & \tau a
    \end{pmatrix}:=
    \exp\begin{pmatrix}
    - i\nu & - \tau\varrho \\
    \varrho & i\nu
    \end{pmatrix} \in SU(2)$, where $ \begin{pmatrix}
    - i\nu & - \tau\varrho \\
    \varrho & i\nu
    \end{pmatrix}\in \mathfrak{su}(2) $, we have $\phi_\upsilon(A) =
    \exp(\ad(0, 0, 0, i\nu,  \varrho, -\tau\varrho)) \in
    (E_8)^{\varepsilon^{}_4}$ (Lemma \ref{lem 9.4} (2)).
\end{proof}


Now, we determine the structure of the group
$(E_8)^{\varepsilon^{}_4}$.

\begin{thm}\label{thm 9.6}
    The group $(E_8)^{\varepsilon^{}_4}$ is isomorphic to the group
    $(SU(2) \times U(1) \times Spin(12))/\allowbreak (\Z_2 \times
    \Z_2), \Z_2=\{(E, 1,1), (E,-1, -\sigma)\},
    \Z_2=\{(E,1,1),(-E, 1, -1) \}${\rm
    :} $(E_8)^{\varepsilon^{}_4} \cong (SU(2) \times
    U(1) \times Spin(12))/(\Z_2 $ $ \times
    \Z_2)$.
\end{thm}
\begin{proof}
    We define a mapping $\varphi_{{}_{\varepsilon^{}_4}}:
    SU(2) \times U(1) \times Spin(12) \to
    (E_8)^{\varepsilon^{}_4}$ by
    \begin{align*}
     \varphi_{{}_{\varepsilon^{}_4}}(A, \theta,
     \beta)=\phi_\upsilon(A)\varphi_{{}_{\nu^{}_4}}(\theta,\beta)(=\varphi(A,\varphi_{{}_{\nu^{}_4}}(\theta,\beta))).
    \end{align*}
   Note that this mapping is the restriction of the mapping $ \varphi:SU(2)\times E_7 \to (E_8)^\upsilon $ defined in the proof of Theorem \ref{thm 7.4}.

    First, we will prove that $ \varphi_{{}_{\varepsilon^{}_4}} $ is well-defined. However, from Propositions \ref{prop 9.3},\ref{prop 9.5}, it is clear that $\varphi_{{}_{\varepsilon^{}_4}}$ is well-defined. Subsequently, we will prove that $ \varphi_{{}_{\varepsilon^{}_4}} $ is a homomorphism. Since the mapping $ \varphi_{{}_{\varepsilon^{}_4}} $ is the restriction of the mapping $ \varphi $ and $ \varphi_{{}_{\nu^{}_4}} $ is a homomorphism (Proposition \ref{prop 9.3}), $ \varphi_{{}_{\varepsilon^{}_4}} $ is a homomorphism.


    Next, we will prove that $\varphi_{{}_{\varepsilon^{}_4}}$ is surjective. Let
    $\alpha \in (E_8)^{\varepsilon^{}_4}$. Then, since the group $ E_8 $ is the simply connected Lie group, both of the groups $ (E_8)^{\varepsilon^{}_4} $ and $ (E_8)^\upsilon $ are connected ((\cite[Preliminaries Lemma 2.2]{miya2}) in \cite{ra}), together with $ (\mathfrak{e}_8)^{\varepsilon^{}_4} \subset (\mathfrak{e}_8)^\upsilon $ (Lemma \ref{lem 9.4} (1), (2)), we have $ \alpha \in (E_8)^{\varepsilon^{}_4} \subset (E_8)^\upsilon $. Hence
    there exist $A \in SU(2)$ and
    $\delta \in E_7$ such that $\alpha =\varphi(A, \delta)$ (Theorem \ref{thm 7.4}).
    Moreover, from the condition
    $\varepsilon^{}_4\alpha{\varepsilon^{}_4}^{-1}=\alpha$, that is, $ \varepsilon^{}_4\varphi(A,\delta){\varepsilon^{}_4}^{-1}=\varphi(A,\delta) $, we have $\varphi(A, \nu^{}_4 \delta {\nu^{}_4}^{-1})=\varphi(A,
    \delta)$. Indeed, from $\varepsilon^{}_4 (0, 0, 0, i\nu, \varrho, -\tau\varrho)=(0, 0, 0, i\nu, $ $\varrho, -\tau\varrho)$ and $\phi_\upsilon(A)=\exp(\ad(0, 0, 0, i\nu, \varrho, -\tau\varrho))$, we have
    $\varepsilon^{}_4\phi_\upsilon(A){\varepsilon^{}_4}^{-1}=\phi_\upsilon (A)$ by a computation similar to that in the proof of Theorem \ref{thm 7.5}.
In addition, $ \varepsilon^{}_4\delta{\varepsilon^{}_4}^{-1}=\nu^{}_4\delta{\nu^{}_4}^{-1} $ follows from $ \delta \in E_7 $. Hence it follows  from
    \begin{align*}
    \varepsilon^{}_4\varphi(A, \delta){\varepsilon^{}_4}^{-1}
    &=\varepsilon^{}_4(\phi_\upsilon (A)\delta )
    {\varepsilon^{}_4}^{-1}
    \\
    &=(\varepsilon^{}_4\phi_\upsilon (A){\varepsilon^{}_4}^{-1})(\varepsilon^{}_4\delta{\varepsilon^{}_4}^{-1})
    \\
    &=\phi_\upsilon(A)(\nu^{}_4\delta {\nu^{}_4}^{-1})
    \\
    &=\varphi(A, \nu^{}_4\delta {\nu^{}_4}^{-1})
    \end{align*}
that $\varphi(A, \nu^{}_4 \delta {\nu^{}_4}^{-1})=\varphi(A, \delta)$.

  Thus, since $ \Ker\varphi=\{(E,1), (-E,-1)\} $, we have the following
    $$
    \left\{\begin{array}{l}
    A = A
    \vspace{1mm}\\
    \nu^{}_4\delta {\nu^{}_4}^{-1} = \delta
    \end{array} \right.
    \qquad   \text{or}\qquad
    \left\{\begin{array}{l}
    A= -A
    \vspace{1mm}\\
    \nu^{}_4\delta {\nu^{}_4}^{-1} = -\delta.
    \end{array} \right.
    $$
    In the latter case, this case is impossible because of $A\not=O$, where $ O $ is the zero matrix.
    In the former case, $ \delta \in (E_7)^{\nu_4} $ follows from the second condition, so that
    there exist $\theta \in U(1)$ and $\beta \in Spin(12)$ such that $\delta=\varphi_{{}_{\nu^{}_4}}(\theta,\beta)$ (Proposition \ref{prop 9.3}). Hence there exist $ A \in SU(2), \theta \in U(1) $ and $ \beta \in Spin(12) $ such that $ \alpha=\varphi(A,\theta,\varphi_{{}_{\nu^{}_4}}(\theta,\beta))=\varphi_{{}_{\varepsilon^{}_4}}(A,\theta,\beta) $.
    The proof of surjective is completed.

    Finally, we will determine $ \Ker\,\varphi_{{}_{\varepsilon^{}_4}} $. From the definition of kernel, we have
    \begin{align*}
    \Ker\,\varphi_{{}_{\varepsilon^{}_4}}
    &=\{(A, \theta, \beta) \in SU(2) \times U(1) \times Spin(12)\,|\,
    \varphi_{{}_{\varepsilon^{}_4}}(A,\theta,\beta)=1  \}
    \\
    &=\{(A, \theta, \beta) \in SU(2) \times U(1) \times
    Spin(12)\,|\, \varphi(A,\varphi_{{}_{\nu^{}_4}}(\theta,\beta))=1 \}.
    \end{align*}
    Here, the mapping $\varphi_{{}_{\varepsilon^{}_4}}$
    is the restriction of the mapping $ \varphi $ and together with $\Ker\,\varphi=\{(E,1), (-E, -1) \}$ (Theorem \ref{thm 7.4}),
    we will find the elements $ (A,\theta, \beta) \in SU(2)\times U(1)\times Spin(12) $ satisfying the following
    \begin{align*}
    \left\{\begin{array}{l}
        A = E
        \vspace{1mm}\\
        \varphi_{{}_{\nu^{}_4}}(\theta,\beta) = 1
    \end{array} \right.
    \qquad   \text{or}\qquad
    \left\{\begin{array}{l}
        A= -E
        \vspace{1mm}\\
        \varphi_{{}_{\nu^{}_4}}(\theta,\beta) = -1.
    \end{array} \right.
    \end{align*}
    In the former case, from $ \Ker\,\varphi_{{}_{\nu^{}_4}}=\{(1,1),(-1,-\sigma)\} $ (Proposition \ref{prop 9.3}), we have the following
    $$
    \left\{\begin{array}{l}
    A = E
    \vspace{1mm}\\
    \theta = 1
    \vspace{1mm}\\
    \beta=1
    \end{array} \right.
    \qquad   \text{or}\qquad
    \left\{\begin{array}{l}
    A= E
    \vspace{1mm}\\
    \theta = -1
    \vspace{1mm}\\
    \beta=-\sigma.
    \end{array} \right.
    $$
    In the latter case, the second condition
    can be rewritten as $\varphi_{{}_{\nu^{}_4}}(\theta,-\beta) = 1$ from $-1 \in z(E_7) $, moreover $ -\beta \in Spin(12) $. Hence, as in the former case, we have the following
    \begin{align*}
     \left\{\begin{array}{l}
         A = -E
         \vspace{1mm}\\
         \theta = 1
         \vspace{1mm}\\
         \beta=-1
     \end{array} \right.
     \qquad   \text{or}\qquad
     \left\{\begin{array}{l}
         A= -E
         \vspace{1mm}\\
         \theta = -1
         \vspace{1mm}\\
         \beta=\sigma.
     \end{array} \right.
    \end{align*}

    Hence we can obtain
    \begin{align*}
    \Ker\,\varphi_{{}_{\varepsilon^{}_4}}&=\{(E,1,1),
    (E,-1, -\sigma), (-E, 1, -1), (-E, -1, \sigma)  \}\\
    &=\{(E,1,1),(E, -1, -\sigma) \} \times \{(E,1,1), (-E, 1, -1) \}\\
    & \cong \Z_2 \times \Z_2.
    \end{align*}

    Therefore we have the required isomorphism
    \begin{align*}
     (E_8)^{\varepsilon^{}_4} \cong (SU(2) \times U(1)
    \times Spin(12))/(\Z_2 \times \Z_2).
    \end{align*}
\end{proof}


\begin{thebibliography}{9}

	\bibitem{Imai} T. Imai and I. Yokota,
	\newblock Simply connected compact simple Lie group $E_{8(-248)}$ of type $E_8$, J. Math. Kyoto Univ. 21(1981), 741-762.





	\bibitem{Jim} J.A. Jim\'{e}nez,
	\newblock Riemannian 4-symmetric spaces, Trans. Amer. Math. Soc. 306 (1988), 715-734.



	\bibitem{miya2} T. Miyashita,
	Realizations of inner automorphisms of order $4$ and fixed points subgroups by them on the connected compact exceptional Lie group $E_8$, Part I, Tsukuba J. Math. 41-1(2017), 91-166.

	\bibitem{miya3} T. Miyashita,
	Realizations of inner automorphisms of order four and fixed points subgroups by them on the connected compact exceptional Lie group $E_8$, Part II, Tsukuba J. Math. 43-1(2019), 1-22.

	\bibitem{miya0} T. Miyashita and I. Yokota,
	\newblock 2-graded decompositions of exceptional Lie algebras
	$\mathfrak{g}$ and group realizations of $\mathfrak{g}_{ev},
	\mathfrak{g}_0$, Part III, $G = E_8$, Japanese J. Math.
	26-1(2000), 31-50.


  \bibitem{ra} Ra{$ \rm\check{s} $}evskii, P.K., A theorem on the connectedness of a subgroup of a simply connected Lie group commuting with any of its automorphisms, Trans. Moscow Math. Soc, 30 (1974), 3-22.





	\bibitem{iy0} I. Yokota,
	\newblock Exceptional Lie groups, arXiv:math/0902.0431vl[mathDG](2009).
\end{thebibliography}
\end{document}